\documentclass{amsart}

\usepackage{amsfonts,amssymb}
\usepackage{array}
\usepackage{enumerate}
\usepackage{todonotes}
\usepackage[all,cmtip]{xy}
\usepackage{float}
\usepackage{calligra}
\usepackage{mathrsfs}
\usepackage[pagebackref,bookmarksopen]{hyperref}
\hypersetup{
	colorlinks=true,
	linkcolor=blue,
	citecolor=magenta,
	urlcolor=cyan,
}
\newcommand{\C}{\mathbb{C}}
\newcommand{\Z}{\mathbb{Z}}
\newcommand{\GLn}[2][\C]{\mathrm{GL}(#2,#1)}

\newcommand{\SO}[1]{\mathrm{SO}(#1)}
\newcommand{\U}[1][2]{\mathrm{U}(#1)}
\newcommand{\SU}{\mathrm{SU}(2)}
\newcommand{\BG}[1][\Gamma]{\mathrm{B}#1}
\newcommand{\EG}[1][\Gamma]{\mathrm{E}#1}
\newcommand{\abs}[1]{\lvert #1\rvert}

\newcommand{\Ss}[1]{\mathcal{O}_{#1}}

\newcommand{\Rs}{\tilde{X}}

\newcommand{\smvee}{\raise0.9ex\hbox{$\scriptscriptstyle\vee$}}
\DeclareMathOperator{\Homs}{\mathscr{H}\text{\kern -3pt {\calligra\large om}}\,}

\DeclareMathOperator{\Ab}{Ab}
\DeclareMathOperator{\Hom}{Hom}
\DeclareMathOperator{\sign}{sign}
\DeclareMathOperator{\rank}{rank}

\newtheorem{Theorem}{Theorem}
\newtheorem{theorem}{Theorem}[section]
\newtheorem{lemma}[theorem]{Lemma}
\newtheorem{corollary}[theorem]{Corollary}
\newtheorem{proposition}[theorem]{Proposition}
\newtheorem{conjecture}[theorem]{Conjecture}

\theoremstyle{definition}
\newtheorem{definition}[theorem]{Definition}
\newtheorem{example}[theorem]{Example}

\newtheorem{remark}[theorem]{Remark}

\numberwithin{equation}{section}

\begin{document}

\title[Classification of indecomposable reflexive modules on $(\mathbb{C}^2/\Gamma,0)$]{Classification of indecomposable reflexive modules on quotient singularities through Atiyah--Patodi--Singer theory}

\author[J.~A.~Arciniega Nev\'arez]{Jos\'e Antonio Arciniega Nev\'arez}
\address{Divisi\'on de Ingenier\'ias, Campus Guanajuato, Universidad de Guanajuato,
Av. Ju\'arez No. 77, Zona Centro, Guanajuato, Gto., Mexico, C.P. 36000}
\curraddr{}
\email{ja.arciniega@ugto.mx}
\thanks{}

\author[J.~L.~Cisneros-Molina]{Jos\'e Luis Cisneros-Molina}
\address{Instituto de Matem\'aticas, Unidad Cuernavaca, Universidad Nacional Aut\'onoma
de M\'exico, Avenida Universidad s/n, Colonia Lomas de Chamilpa, Cuernavaca,
Morelos, Mexico}
\curraddr{}
\email{jlcisneros@im.unam.mx}
\thanks{}

\author[A.~Romano Vel\'azquez]{Agust\'in Romano Vel\'azquez}
\address{Instituto de Matem\'aticas, Unidad Cuernavaca, Universidad Nacional Aut\'onoma
de M\'exico, Avenida Universidad s/n, Colonia Lomas de Chamilpa, Cuernavaca,
Morelos, Mexico}
\curraddr{}
\email{agustin.romano@im.unam.mx}
\thanks{}

\subjclass[2010]{Primary }


\date{}

\dedicatory{}

\begin{abstract}
In \cite{Esnault:RMQSS} Esnault asked whether on a general quotient surface singularity the rank and the first Chern class distinguish isomorphism classes of indecomposable reflexive modules.
Wunram gave a contraexample in \cite{Wunram:RMOQSS} showing two different full shaves on a quotient singularity, with the same rank and the same first Chern class.
In this article, we prove that irreducible reflexive modules over quotient surface singularities are determined by the rank, first Chern class and the Atiyah-Patodi-Singer $\tilde{\xi}$-invariant
\cite{Atiyah-Patodi-Singer:SARGIII},
except for the case of rank $2$ indecomposable reflexive modules over dihedral quotient surface singularities $\mathbb{D}_{n,q}$ with $\gcd(m,2)=2$, which we conjecture to follow the same pattern.
To prove the classification theorem, first we prove that every spherical $3$-manifold with non-trivial fundamental group appears as the link of a quotient surface singularity, and
that indecomposable flat vector bundles over spherical $3$-manifolds are classified by their rank, first and second Cheeger-Chern-Simons classes, with the exception of the aforementioned case.
\end{abstract}

\maketitle

\section{Introduction}

McKay correspondence was presented by John McKay in \cite{McKay:GSFG}, it establishes a bijection between the non-trivial irreducible representations of a finite subgroup $\Gamma$ of $\mathrm{SL}(2,\mathbb{C})$ and the irreducible components of the exceptional divisor in the minimal resolution of the associated Kleinian singularity, $X = \mathbb{C}^2 / \Gamma$.

Gonzalez-Sprinberg and Verdier~\cite{Gonzalez-Sprinberg-Verdier:McKay} provided a geometric construction, case by case, of this correspondence as follows:
using the Riemann-Hilbert correspondence, to every non-trivial irreducible representation $\rho$ of $\Gamma$ they associated a non-trivial indecomposable reflexive $\mathcal{O}_{X}$-module $M$,
the pull-back $\pi^*M/\mathrm{tor}$ of $M$ on the minimal resolution $\pi\colon\tilde{X}\to X$, called the \textit{full sheaf associated to $M$}, is locally free and they proved that its first Chern class $c_1(\pi^*M/\mathrm{tor})$ is the Poincar\'e dual of a curvette intersecting exactly one irreducible component of the exceptional divisor, and that it determines the module $M$ and the representation $\rho$.
Knörrer \cite{Knorrer:GRRRDP} later gave a group-theoretical interpretation of this geometrical construction, and Artin and Verdier~\cite{Artin-Verdier:RMORDP} provided an elegant theoretical proof of the fact that the first Chern class determines the module $M$.
However, Knörrer and Verdier raised the question of whether an indecomposable reflexive module on a general quotient surface singularity is similarly determined by its first Chern class.
In \cite{Esnault:RMQSS} Esnault gave a negative answer to this question, giving an example of a quotient singularity for which there are an
invertible full sheaf and a rank $2$ indecomposable full sheaf having the same first Chern class; and
asked whether the Chern polynomial (that is, the first Chern class and the rank) distinguishes isomorphism classes of indecomposable reflexive modules.
Wunram \cite{Wunram:RMOQSS} answered this question negatively constructing two different full sheaves on a quotient singularity, with the same rank and the same first Chern class, leaving the classification of indecomposable reflexive modules on quotient singularities unresolved. The main difficulty is that the classification problem of full sheaves requires an invariant that determines the holomorphic structure. In Section \ref{sec.Classification} we prove that it is impossible to classify full sheaves only with topological invariants. It is important to remark that even though there is no moduli for line bundles in quotient singularities, we have found moduli for full sheaves of rank higher than one. This is the main difficult part of this problem.

In this article, we give an (almost) complete classification of indecomposable reflexive modules on quotient singularities, resolving a problem that has remained open for 37 years.

Quotient surface singularities $(X,x)$ are isomorphic to germs $(\mathbb{C}^2/\Gamma,0)$, where $\Gamma$ is a small finite subgroup of $\mathrm{GL}(2,\mathbb{C})$, that is, its elements does not have $1$ as an eigenvalue. If $\Gamma$ and $\Gamma'$ are two small finite subgroups of $\mathrm{GL}(2,\mathbb{C})$, then $\mathbb{C}^2/\Gamma$ and $\mathbb{C}^2/\Gamma'$ are analytically isomorphic if $\Gamma$ and $\Gamma'$ are conjugate. Hence, quotient surface singularities are classified by conjugacy classes of small finite subgroups of $\mathrm{GL}(2,\mathbb{C})$.
The list of such groups was given by Brieskorn \cite[Satz~2.9]{Brieskorn:RSKF} with the notation used in \cite[p.~57]{DuVal:HQR} which is difficult to work with.
Riemenschneider \cite[p.~38]{Riemenschneide:IEUGL2C} listed the matrices which generate these groups, but he does not give relations, so one does not have presentations for the small finite subgroups of $\GLn{2}$.

Every finite subgroup of $\GLn{2}$ is conjugate to a finite subgroup of $\U$, thus, these subgroups fix the $3$-dimensional spheres in $\C^2$ centered at the origin.
This implies that the link $L$ of the quotient singularity $\mathbb{C}^2/\Gamma$ is diffeomorphic to the spherical $3$-manifold $\mathbb{S}^3/\Gamma$.
Spherical $3$-manifolds are compact oriented $3$-manifolds with finite fundamental group, they are of the form $\mathbb{S}^3/G$ where $G$ is a finite subgroup of $\SO{4}$
acting freely and isometrically on $\mathbb{S}^3$. By Hopf \cite{Hopf:ZCKR} the list of such groups is known and they are given by presentations.
Since the links of quotient singularities are spherical $3$-manifolds, in order to have presentations for the small finite subgroups of $\U$, one needs to see to which groups
they correspond in the list of finite subgroups of $\SO{4}$ which act freely and isometrically on $\mathbb{S}^3$. It turns out that basically the two list coincide, so in Section~\ref{sec:groups.lists} we prove the following result.

\begin{Theorem}\label{Thm1}
Every spherical $3$-manifold with non-trivial fundamental group appears as the link of a quotient surface singularity.
\end{Theorem}

The finite subgroups of $\SO{4}$ acting freely and isometrically on $\mathbb{S}^3$ consist of the trivial group, the binary polyhedral groups, two families of groups denoted by $D_{2^{k+1}(2r+1)}$ and $P'_{8\cdot 3^k}$, and the direct product of any of them with a cyclic group of relatively prime order. The irreducible representations of cyclic and binary polyhedral groups are known, in Section~\ref{sec:irr.rep} we give the irreducible representations of the groups $D_{2^{k+1}(2r+1)}$ and $P'_{8\cdot 3^k}$ which, as far as we know, have not previously been detailed in the literature.

Our next goal is to classify flat vector bundles over spherical $3$-manifolds using  Cheeger-Chern-Simons classes which are secondary characteristic classes.
Isomorphism classes of rank $n$ flat vector bundles over a compact smooth manifold are in one-to-one correspondence with $n$-dimensional representations of the fundamental group of the manifold.
In Section~\ref{sec:CCS} we recall some results on flat vector bundles, Chern and Cheeger-Chern-Simons classes and CCS-numbers of representions and their properties. In particular,
our result \cite[Theorems~1 \& 2]{arciniegaetal:CCSC} that for a rational homology $3$-sphere $L$ (e.g. a spherical $3$-manifold) and a representation $\rho\colon\pi_1(L)\to\GLn{\C}$ of its fundamental group, the second CCS-number of $\rho$, that is, the second Cheeger-Chern-Simons class $\widehat{c}_{\rho,2}$ evaluated on the fundamental class $[L]$, is given in terms of the Atiyah-Patodi-Singer $\tilde{\xi}$-invariant of the Dirac operator $D$ on $L$, twisted by the representation $\rho$ and its determinant $\det(\rho)$, namely, $\widehat{c}_{\rho,2}([L])=\tilde\xi_{\rho}(D)-\tilde\xi_{\det(\rho)}(D)$. 
Using the computations in \cite{arciniegaetal:CCSC} of the first and second CCS-numbers of representations of binary polyhedral groups we prove
the following classification result in Section~\ref{sec:Class.FVB}:

\begin{Theorem}\label{Thm2}
Let $M$ be a spherical $3$-manifold. Indecomposable flat vector bundles over $M$ are classified by their rank, first and second Cheeger-Chern-Simons classes.
\end{Theorem}

Theorem~\ref{Thm2} remains unproven for rank $2$ flat vector bundles over spherical $3$-manifolds $\mathbb{S}^3/\Gamma$ when $\Gamma = D_{2^{k+1}(2r+1)} \times \mathsf{C}_l$,
where $\mathsf{C}_l$ is a cyclic group of order relatively prime to $2^{k+1}(2r+1)$. We conjecture that Theorem~\ref{Thm2} holds in these cases (see Conjecture~\ref{conj:Dk2r.2d}).

Finally, in Section~\ref{sec:sing}, we revisit basics on reflexive modules, full sheaves and topological invariants of quotient surface singularities.
In particular, given a quotient singularity $(X,x)=(\C^2 / \Gamma,0)$ and its minimal resolution $\pi\colon\tilde{X}\to X$, by results of Gustavsen-Ile \cite{Gustavsen-Ile:RMNSSRLFG} and Esnault~\cite{Esnault:RMQSS} there is a one-to-one correspondence between representations of $\Gamma$, reflexive $\Ss{X}$-modules and full sheaves on $\tilde{X}$.
A particular case of \cite[Theorem~6.19]{arciniegaetal:CCSC} allows us to compute the invariant $\tilde{\xi}_\rho(D)$ on the minimal resolution $\tilde{X}$ using the full sheaf $\mathcal{M}$ corresponding to the representation $\rho$ under the aforementioned bijections, so we can see the invariant $\tilde{\xi}_\rho(D)$ as an invariant of $\mathcal{M}$.
Using Theorem~\ref{Thm2} we prove the main classification result:

\begin{Theorem}\label{Thm3}
Let $(X,x)$ be the germ of a quotient surface singularity. Isomorphism classes of indecomposable reflexive $\Ss{X}$-modules are determined by their rank, first Chern class and $\tilde{\xi}$-invariant.
\end{Theorem}

Since Theorem~\ref{Thm2} underlies Theorem~\ref{Thm3}, the case of rank $2$ full sheaves on quotient singularities $(\mathbb{C}^2/\Gamma, 0)$ with $\Gamma = D_{2^{k+1}(2r+1)} \times \mathsf{C}_l$ remains open, pending the resolution of Conjecture~\ref{conj:Dk2r.2d}. In Riemenschneider's notation, these groups correspond to the dihedral groups $\mathbb{D}_{n,q}$ with $\gcd(m,2)=2$ (see Theorem~\ref{thm:D.T}).

We include an appendix on Chern classes of tensor products of representations.

\section{Spherical \texorpdfstring{$3$}{3}-manifolds and quotient surface singularities}\label{sec:groups.lists}

In this section we prove that every spherical $3$-manifold appears as the link of a quotient surface singularity.

\subsection{Spherical \texorpdfstring{$3$}{3}-manifolds}

Let $\mathbb{S}^3$ be the unit sphere in $\mathbb{R}^4$. The group $\SO{4}$ of rotations of $\mathbb{R}^4$ fixes $\mathbb{S}^3$.
A \textit{spherical $3$-manifold} $M$ is a manifold of the form $M=\mathbb{S}^3/G$, where $G$ is a finite subgroup of $\SO{4}$ acting freely by rotations on  $\mathbb{S}^3$.
All such manifolds are prime, orientable and closed.
The projection $\mathbb{S}^3\to M$ is the universal cover and therefore the fundamental group $\pi_1(M)$ of $M$ is isomorphic to $G$.
The elliptization conjecture by Thurston \cite[p.~28]{Thurston:GT3M}, proved by Perelman \cite[Theorem~1.7.3]{Aschenbrenner-etal:3MG}, states that conversely all compact oriented $3$-manifolds with finite fundamental group are spherical $3$-manifolds.
Thus, spherical $3$-manifolds are classified by the finite subgroups of $\SO{4}$ acting freely and isometrically on $\mathbb{S}^3$.
In order to give a description of such groups, suitable for our purposes, we need to recall the binary polyhedral  groups.
Let $\mathbb{H}$ denote the algebra of quaternions, we can identify $\mathbb{S}^3$ with the group of unit quaternions.
Let $\langle 2,s,t\rangle$ denote the group given by the presentation
\begin{equation}\label{eq:2st}
\langle 2,s,t\rangle=\langle b,c\,|\,(bc)^2=b^s=c^t\rangle.
\end{equation}
The binary polyhedral groups are the finite subgroups of $\mathbb{S}^3$ given in Table~\ref{tbl:BPG}:\\
\begin{table}[H]
\begin{tabular}{llll}
\textbf{Binary dihedral groups} & $\mathsf{BD}_{2t}=\langle 2,2,t\rangle$, & order $4t$, & $t\geq2$,\\
\textbf{Binary tetrahedral group} & $\mathsf{BT}=\langle 2,3,3\rangle$, & order $24$, &\\
\textbf{Binary octahedral group} & $\mathsf{BO}=\langle 2,3,4\rangle$, & order $48$, &\\
\textbf{Binary icosahedral group} & $\mathsf{BI}=\langle 2,3,5\rangle$, & order $120$. &\\[10pt]
\end{tabular}
\caption{Binary polyhedral groups.}\label{tbl:BPG}
\end{table}
The binary polyhedral groups, together with the finite cyclic groups give all the finite subgroups of $\mathbb{S}^3$ \cite{Cayley:Poly,Klein:Icosahedron,Lamotke:RSIS}.
If $\Gamma$ is a finite subgroup of $\mathbb{S}^3$ then it acts on $\mathbb{S}^3$ freely and isometrically by left multiplication.

The following Lemma gives another presentation for the groups $\mathsf{BT}$, $\mathsf{BO}$ and $\mathsf{BI}$.
Let $P_{\frac{24n}{6-n}}$ denote the group given by the presentation (see \cite[p.~628]{Milnor:GASNWFP})
\begin{equation}\label{eq:Pn}
P_{\frac{24n}{6-n}}=\langle x,y\,|\,x^2=(xy)^3=y^n,\ x^4=1\rangle, \quad\text{with $n=3,4,5$.}
\end{equation}

\begin{lemma}\label{lem:P=BP}
The groups $\langle 2,3,n\rangle$ and $P_{\frac{24n}{6-n}}$ are isomorphic when $n=3,4,5$.
\end{lemma}

\begin{proof}
In \cite[\S4]{Coxeter:BPGOGQG} Coxeter proved that the relations $z=(bc)^2=b^3=c^n$ of the group $\langle 2,3,n\rangle$, with $n=3,4,5$, imply that
\begin{equation}\label{eq:ord2}
 z^2=1.
\end{equation}
This implies that $c^{-n}=c^n$. Set $x=bc$ and $y=c^{-1}$. Thus we have
\begin{equation*}
 x^2=(bc)^2=b^3=(bcc^{-1})^3=(xy)^3=c^n=c^{-n}=y^n.
\end{equation*}
Also by \eqref{eq:ord2} we have that $x^4=(bc)^4=z^2=1$. 
\end{proof}

By \cite{Hopf:ZCKR} a finite subgroup of $\SO{4}$ acting freely on $\mathbb{S}^3$ is isomorphic to one of the groups given in Theorem~\ref{thm:fsgfi} (the list is taken from \cite[\S1.7]{Aschenbrenner-etal:3MG}, see also \cite[\S3]{Milnor:GASNWFP}):
\begin{theorem}\label{thm:fsgfi}
The finite subgroups of $\SO{4}$ acting freely and isometrically on $\mathbb{S}^3$:
\begin{enumerate}[(1)]
 \item The trivial group,\label{it:ffi.triv}
 \item $Q_{4t}:=\langle x,y\mid x^2=(xy)^2=y^t\rangle\cong\mathsf{BD}_{2t}$ where $t\geq2$,\label{it:ffi.Q}
 \item $P_{24}:=\langle x,y\,|\,x^2=(xy)^3=y^3,\ x^4=1\rangle\cong\mathsf{BT}$,\label{it:ffi.P24}
 \item $P_{48}:=\langle x,y\,|\,x^2=(xy)^3=y^4,\ x^4=1\rangle\cong\mathsf{BO}$,\label{it:ffi.P48}
 \item $P_{120}:=\langle x,y\,|\,x^2=(xy)^3=y^5,\ x^4=1\rangle\cong\mathsf{BI}$,\label{it:ffi.P120}
 \item $D_{2^{k+1}(2r+1)}:=\langle x,y\mid x^{2^{k+1}}=1, y^{2r+1}=1, xyx^{-1}=y^{-1}\rangle$, where $k>1$, $r\geq1$,\label{it:ffi.Dmk}
 \item $P'_{8\cdot 3^k}:=\langle x,y,z\mid x^2=(xy)^2=y^2,zxz^{-1}=y, zyz^{-1}=xy, z^{3^k}=1\rangle$, where $k\geq2$,\label{it:ffi.Pk}
 \item the direct product of any of the above groups with a cyclic group of relatively prime order.\label{it:Gammaxcyc}
\end{enumerate}
\end{theorem}

\begin{remark}\label{rem:isos}
The subindex in the name of the group is the order of the group.
The group $D_{4(2r+1)}$ is isomorphic to the group $\mathsf{BD}_{2(2r+1)}$ for every $r\geq1$ \cite[p.~628]{Milnor:GASNWFP}, for this reason in \eqref{it:ffi.Dmk} we take $k>1$. Using, respectively, the presentations \eqref{it:ffi.Dmk} and $\langle 2,2,2r+1\rangle$ for $D_{4(2r+1)}$ and
$\mathsf{BD}_{2(2r+1)}$, an isomorphism and its inverse are
\begin{align}\label{eq:Dmk.BDk}
 x&\mapsto b,\quad y\mapsto c^2, &\text{and} & & b&\mapsto x,\quad c\mapsto y^{r+1}x^2.
\end{align}
The group $P'_{24}$ is isomorphic to the group $P_{24}$.
Using the presentation \eqref{it:ffi.Pk} with $k=1$ for $P'_{24}$ and presentation \eqref{it:ffi.P24} for $P_{24}$, using capital letters for the generator of $P'_{24}$ to distinguish them from the generators of $P_{24}$, an isomorphism and its inverse are given by
\begin{align}\label{eq:P24.P24}
X&\mapsto yxy^{-1},\quad Y\mapsto x,\quad Z\mapsto y^2,  &\text{and} & & x&\mapsto Y,\quad y\mapsto X^{-1}Z^{-1}Y^{-1}.
\end{align}
In \eqref{it:ffi.Q} the presentation of $Q_{4t}$ is precisely the presentation $\langle 2,2,t\rangle$ of $\mathsf{BD}_{2t}$.
In \eqref{it:ffi.P24}, \eqref{it:ffi.P48} and \eqref{it:ffi.P120} the isomorphisms are given by Lemma~\ref{lem:P=BP}.
\end{remark}

\subsection{Quotient surface singularities}\label{ssec:qss}

Let $\Gamma$ be a finite subgroup of $\GLn{2}$. Cartan proved in \cite{Cartan:QEAGA} that the quotient space $\mathbb{C}^2/\Gamma$ is a normal analytic surface with an isolated singularity
by embedding it in $\mathbb{C}^q$, for some $q\in\mathbb{N}$, with image a normal analytic surface.
An element of $\GLn{2}$ is a \textit{pseudo-reflexion} if it fixes a line, that is, if it has $1$ as an eigenvalue.
In \cite{Prill:LCQCMDG} Prill call a subgroup $\Gamma$ of $\GLn{2}$ \textit{small} if no $g\in \Gamma$  is a pseudo-reflexion.
A \textit{quotient surface singularity} is a germ $(X,x)\cong(\mathbb{C}^2/\Gamma,0)$ with $\Gamma$ a small finite subgroup of $\GLn{2}$.
We say that two quotient singularities are \textit{equivalent} if they are analytically isomorphic. Also in \cite{Prill:LCQCMDG} Prill
proved that if $\Gamma$ and $\Gamma'$ are two small finite subgroups of $\GLn{2}$ then $(\mathbb{C}^2/\Gamma,0)$ and $(\mathbb{C}^2/\Gamma',0)$ are equivalent if and only if $\Gamma$ and $\Gamma'$ are conjugate.
Any finite subgroup of $\GLn{2}$ is conjugate to a finite subgroup of $\U$ \cite[Lemma~A.17]{Pe:NPQSS}, thus,
the classification of quotient surface singularities is equivalent to the classification of small finite subgroups of $\U$.
The list of such groups was given by Brieskorn in \cite[Satz~2.9]{Brieskorn:RSKF} (see also \cite{Pe:NPQSS}).
All the finite subgroups of $\SU$ are small, and since $\SU\cong\mathbb{S}^3$, they are the binary polyhedral groups given in Table~\ref{tbl:BPG} together with the finite cyclic groups.
The surface quotient singularities $\mathbb{C}^2/\Gamma$ with $\Gamma$ a finite subgroup of $\SU$ are the rational double point singularities.

Using the embedding of a quotient surface singularity $(X,x)\cong(\mathbb{C}^2/\Gamma,0)$ in $\mathbb{C}^q$, the \textit{link} $L$ of $(X,x)$
is defined as the transverse intersection $L=X\cap\mathbb{S}^{2q-1}_\epsilon$, where $\mathbb{S}^{2q-1}_\epsilon=\{\,z\in\mathbb{C}^q\mid |z|=\epsilon\,\}$ and $\epsilon>0$ is small enough.
Thus the link $L$ is a smooth compact oriented $3$-manifold. The diffeomorphism type of $L$ does not depend on the embedding nor on $0<\epsilon\ll 1$.
Since we can take $\Gamma\leq \U$, it preserves the unit sphere $\mathbb{S}^3$ in $\mathbb{C}^2$, so one can get the link taking the quotient of $\mathbb{S}^3$ by $\Gamma$, that is $L=\mathbb{S}^3/\Gamma$.
Therefore, the links of quotient surface singularities are spherical $3$-manifolds.
In fact, by Brieskorn \cite[Satz~2.8]{Brieskorn:RSKF} quotient surface singularities are the normal surface singularities whose link has finite fundamental group.
To see that every spherical $3$-manifold appears as the link of a quotient surface singularity we need to see that the list of small finite subgroups of
$\U$ coincides with the list of groups given in Theorem~\ref{thm:fsgfi}.
We need some preparation in order to describe the list of small finite subgroups of $\U$.

\subsubsection{Cyclic subgroups of $\U$}
Every cyclic subgroup of $\U$ is conjugate to one of the following cyclic groups
\begin{equation}\label{eq:cyclic.U2}
 \mathbb{C}_{n,q}=\left\langle
 \begin{pmatrix}
  \zeta_n & 0\\
  0 & \zeta_n^q
 \end{pmatrix} \right\rangle\qquad 0<q<n,\quad (n,q)=1,
\end{equation}
where $\zeta_n$ is the primitive root of unity $e^{\frac{2\pi i}{n}}$. The groups $\mathbb{C}_{p,q}$ with $0<q<n$ and $(n,q)=1$ are small if $n>1$.

\begin{remark}\label{rem:Cnp.Cmq}
Two groups $\mathbb{C}_{n,q}$ and $\mathbb{C}_{n',q'}$ with $(n,q)=(n',q')=1$ are conjugate if and only if $n=n'$ and either $q=q'$ or $qq'\equiv 1 \mod n$ \cite[Lemma~A.12]{Pe:NPQSS}.
\end{remark}

\subsubsection{Extension of binary polyhedral groups}
Coxeter and Moser in \cite[p.~71]{Coxeter-Moser:GRDG} give an extension of the binary polyhedral groups, defined by the presentation
\begin{equation*}
 \langle 2,s,t\rangle_m=\langle (bc)^2=b^s=c^t=z^m,\ z\rightleftarrows bc\rangle,
\end{equation*}
where $z\rightleftarrows bc$ means that $z$ commutes with $b$ and $c$.
They are denoted by $\langle 2,s,t\rangle_m$, since $\langle 2,s,t\rangle_1=\langle 2,s,t\rangle$. We have that \cite[(6.59)]{Coxeter-Moser:GRDG}
\begin{equation}\label{eq:prod}
 \langle 2,s,t\rangle_m\cong\langle 2,s,t\rangle\times\mathsf{C}_m,\qquad \text{when $m$ is odd,}
\end{equation}
where $\mathsf{C}_m$ is the cyclic group of order $m$.

\subsubsection{Riemenschneider's matrices}
We shall follow Riemenschneider's notation in \cite[p.~38]{Riemenschneide:IEUGL2C} where the generator of the small finite subgroups of $\U$ are given.
Let $\zeta_k=e^{\frac{2\pi i}{k}}$, $k\in\mathbb{N}$. Set
\begin{align*}
 \phi_k&=\begin{pmatrix}
          \zeta_k  & 0 \\ 0 & \zeta_k
         \end{pmatrix},  &
 \psi_k&=\begin{pmatrix}
          \zeta_k  & 0 \\ 0 & \zeta_k^{-1}
         \end{pmatrix},  &
 \eta&=\frac{1}{\sqrt{2}}\begin{pmatrix}
                         \zeta_8 &    \zeta_8^3 \\ \zeta_8 & \zeta_8^7
                         \end{pmatrix},  \\
 \tau&=\begin{pmatrix}
        0 & i \\ i & 0
       \end{pmatrix}, &
 \omega&=\begin{pmatrix}
          \zeta_5^3 & 0 \\ 0 & \zeta_5^2
         \end{pmatrix}, &
 \iota&=\frac{1}{\sqrt{5}}\begin{pmatrix}
                           \zeta_5^4-\zeta_5 & \zeta_5^2-\zeta_5^3\\ \zeta_5^2-\zeta_5^3 & \zeta_5-\zeta_5^4
                          \end{pmatrix},\\
 \sigma&=\begin{pmatrix}
          0 & -1 \\ 1 & 0
         \end{pmatrix}. & & & &
\end{align*}

\subsubsection{Finite small subgroups of $\U$}

The list of the finite subgroups of $\U$ is given by the first $9$ families of groups in \cite[p.~98]{Coxeter:RCP} (see also \cite[p.~57]{DuVal:HQR}), among these groups,
Brieskorn in \cite[Satz~2.9]{Brieskorn:RSKF} found the conditions in order to have \textbf{small} subgroups obtaining only 7 families. In \cite[p.~38]{Riemenschneide:IEUGL2C} Riemenschneider gives the generators of these groups and divides Brieskorn's $7$ families into $5$ types. Combining the information of \cite[p.~98]{Coxeter:RCP}, \cite[Satz~2.9]{Brieskorn:RSKF} and \cite[p.~38]{Riemenschneide:IEUGL2C} we get:

\begin{theorem}\label{thm:fssu2}
Each small finite subgroup of $\U$ is conjugate to one of the following groups:
\begin{align*}
\intertext{\textbf{Cyclic groups.} Let $0<q<n$ with $\gcd(n,q)=1$,}
 \mathbb{C}_{n,q}&=\left\langle\left(\begin{smallmatrix}
                                                 \zeta_n & 0 \\ 0 & \zeta_n^q
                                                \end{smallmatrix}\right)\right\rangle,\quad \text{order $n$.}\\
\intertext{\textbf{Dihedral groups.} Let $0<q<n$ with $\gcd(n,q)=1$ and $m=n-q$,}
 \mathbb{D}_{n,q}&=\begin{cases}
                   \langle 2,2,q\rangle_m\cong\mathsf{BD}_{2q}\times\mathsf{C}_m\cong\langle\psi_{2q},\tau,\phi_{2m}\rangle, & \gcd(m,2)=1,\\
                   \langle\psi_{2q},\tau\phi_{4m}\rangle, & \gcd(m,2)=2,
                  \end{cases}\quad \text{order $4qm$.}\\
\intertext{\textbf{Tetrahedral groups.}}
 \mathbb{T}_m&=\begin{cases}
               \langle 2,3,3\rangle_m\cong\mathsf{BT}\times\mathsf{C}_m\cong\langle\psi_4,\tau,\eta,\phi_{2m}\rangle, & \gcd(m,6)=1,\\
               \langle\psi_4,\tau,\eta\phi_{6m}\rangle, & \gcd(m,6)=3,\\
              \end{cases}\quad \text{order $24m$.}\\
\intertext{\textbf{Octahedral groups.}}
 \mathbb{O}_m&=\langle 2,3,4\rangle_m\cong\mathsf{BO}\times\mathsf{C}_m\cong\langle\psi_8,\tau,\eta,\phi_{2m}\rangle,\qquad \gcd(m,6)=1,\quad \text{order $48m$.}\\
\intertext{\textbf{Icosahedral groups.}}
 \mathbb{I}_m&=\langle 2,3,5\rangle_m\cong\mathsf{BI}\times\mathsf{C}_m\cong\langle\sigma,\omega,\iota,\phi_{2m}\rangle,\qquad \gcd(m,30)=1,\quad \text{order $120m$.}
\end{align*}
\end{theorem}

\begin{remark}\label{rem:rpo}
Note the following:
\begin{itemize}
\item For the Dihedral groups, the condition $\gcd(n,q)=1$ implies $\gcd(m,q)=1$,

when $m$ is odd, this implies that $\gcd(m,4q)=1$, hence, by \eqref{eq:prod} we obtain the product of
$\mathsf{BD}_{2q}$ with a cyclic group of order $m$, relatively prime to $4q$, the order of $\mathsf{BD}_{2q}$.
\item For the Tetrahedral groups with $\gcd(m,6)=1$, this implies that $m$ is odd and $\gcd(m,24)=1$, hence, by \eqref{eq:prod} we obtain the product of
$\mathsf{BT}$ with a cyclic group of order $m$, relatively prime to $24$, the order of $\mathsf{BT}$.
\item For the Octahedral groups, $\gcd(m,6)=1$ implies that $m$ is odd and $\gcd(m,48)=1$, hence, by \eqref{eq:prod} we obtain the product of
$\mathsf{BO}$ with a cyclic group of order $m$, relatively prime to $48$, the order of $\mathsf{BO}$.
\item For the Icosahedral groups, the condition on $\gcd(m,30)=1$ implies that $m$ is odd and $\gcd(m,120)=1$, hence, by \eqref{eq:prod} we obtain the product of
$\mathsf{BI}$ with a cyclic group of order $m$, relatively prime to $120$, the order of $\mathsf{BI}$.
\end{itemize}
\end{remark}

By Remark~\ref{rem:rpo} we have that families \eqref{it:ffi.triv}, \eqref{it:ffi.Q}, \eqref{it:ffi.P24}, \eqref{it:ffi.P48} and \eqref{it:ffi.P120} of groups in Theorem~\ref{thm:fsgfi} and their direct products with cyclic groups of relatively prime order correspond, respectively, to the cyclic groups $\mathbb{C}_{n,q}$, the dihedral groups $\mathbb{D}_{n,q}$ with $\gcd(n-q,2)=1$, the
tetrahedral groups $\mathbb{T}_m$ with $\gcd(m,6)=1$, the octahedral groups $\mathbb{O}_m$ and the icosahedral groups $\mathbb{I}_m$ in Theorem~\ref{thm:fssu2}.
It remains to prove that the families \eqref{it:ffi.Dmk} and \eqref{it:ffi.Pk} and their direct products with cyclic groups of relatively prime order, correspond, respectively, to the dihedral groups $\mathbb{D}_{n,q}$ with $\gcd(n-q,2)=2$ and the tetrahedral groups $\mathbb{T}_m$ with $\gcd(m,6)=3$.

\begin{theorem}\label{thm:D.T}
One has the following isomorphisms:
\begin{enumerate}[(I)]
\item Let $0<q<n$ with $\gcd(n,q)=1$ and $m=n-q=2^{k-1}l$ with $l$ odd and $k\geq2$. Then $\gcd(2^{k+1}q,l)=1$ and\label{it:T.A}
\begin{equation*}
\mathbb{D}_{n,q}\cong D_{2^{k+1}\cdot q}\times\mathsf{C}_l.
\end{equation*}
\item Let $\gcd(m,6)=3$. Write $m=3^{k-1}l$ with $l$ odd, $\gcd(3,l)=1$ and $k\geq2$. Then $\gcd(8\cdot3^k,l)=1$ and\label{it:T.B}
\begin{equation*}
\mathbb{T}_m\cong P'_{8\cdot 3^k}\times\mathsf{C}_l.
\end{equation*}
\end{enumerate}
\end{theorem}

\begin{proof}
\eqref{it:T.A} Let $\mathbb{I}$ denote the $2\times 2$ identity matrix.
Since $\gcd(n,q)=1$ and $m=n-q\equiv 0\mod 2$, both $n$ and $q$ must be odd.
We claim that $\gcd(2^{k+1}q,l)=1$, suppose $t$ divides both $2^{k+1}q$ and $l$, since $l$ is odd, $t$ must be odd as well, therefore $t$ does not divide $2^{k+1}$ and must divide $q$. We have that $n=2^{k+1}l+q$ and thus $t$ divides $n$, but since $\gcd(n,q)=1$ we have $t=1$.
Since $\phi_{4m}$ is a diagonal matrix we have the commutators
\begin{equation}\label{eq:rel.mat}
	[\phi_{4m},\psi_{2q}]=1,\quad [\phi_{4m},\tau]=1,
\end{equation}
and since $\tau^{-1}=-\tau$ it is easy to see that
\begin{equation}\label{eq:tau.psi}
\tau\psi_{2q}\tau^{-1}=\tau^{-1}\psi_{2q}\tau=\psi_{2q}^{-1}.
\end{equation}
Since $\tau$ and $\phi_{4m}$ commute, $\tau$ has order $4$ and $m$ is even, $\tau\phi_{4m}$ has order $4m=2^{k+1}l$ and $(\tau\phi_{4m})^{2m}=-\mathbb{I}$:
\begin{equation}\label{eq:tau.phi.2m}
 (\tau\phi_{4m})^{2^k\cdot l}=(\tau\phi_{4m})^{2m}=\tau^{2m}\phi_{4m}^{2m}=(\tau^2)^m(-\mathbb{I})=(-\mathbb{I})^m(-\mathbb{I})=(-\mathbb{I})^{m+1}=-\mathbb{I}.
\end{equation}
We also have that $\psi_{2q}=\psi_{q}^{\frac{q+1}{2}}(\tau\phi_{4m})^{2m}$, hence
\begin{equation*}
\mathbb{D}_{n,q}=\langle\psi_{2q},\tau\phi_{4m}\rangle=\langle\psi_{q},\tau\phi_{4m}\rangle=\{\psi_{q}^\beta(\tau\phi_{4m})^\gamma\mid \beta=1,\dots,q; \gamma=1,\dots,4m\}.
\end{equation*}
By \eqref{it:ffi.Dmk} we have the presentation of the direct product
\begin{equation}\label{eq:prod.pres}
D_{2^{k+1}\cdot q}\times\mathsf{C}_l=\langle x,y,z\mid x^{2^{k+1}}=1, y^{q}=1, xyx^{-1}=y^{-1}, z^l=1, z\rightleftarrows xy\rangle
\end{equation}
where $z\rightleftarrows xy$ means that $z$ commutes with $x$ and $y$. The isomorphism is given by
\begin{gather*}
	\Psi\colon D_{2^{k+1}\cdot q}\times\mathsf{C}_l\to \mathbb{D}_{n,q}\\
	x\mapsto (\tau\phi_{2^{k+1}l})^l,\quad y\mapsto \psi_{q},\quad z\mapsto (\tau\phi_{2^{k+1}l})^{2^{k+1}},
\end{gather*}
which is well-defined since  the images $\Psi(x)$, $\Psi(y)$ and $\Psi(z)$ satisfy the relations of $D_{2^{k+1}\cdot q}\times\mathsf{C}_l$ given in \eqref{eq:prod.pres}:
by \eqref{eq:tau.phi.2m} $\Psi(x)=(\tau\phi_{2^{k+1}l})^l$ has order $2^{k+1}$ and $\Psi(z)=(\tau\phi_{2^{k+1}l})^{2^{k+1}}$ has order $l$ and obviously they commute; $\Psi(y)=\psi_{q}$ has order $q$; by \eqref{eq:tau.psi} $\tau\psi_{q}\tau^{-1}=\tau\psi_{2q}^2\tau^{-1}=\psi_{2q}^{-2}=\psi_{q}^{-1}$ and $\tau\psi_{q}^{-1}\tau^{-1}=\psi_{q}$ and  by \eqref{eq:rel.mat} and \eqref{eq:tau.psi} and the fact that $l$ is odd, we have
\begin{equation*}
(\tau\phi_{2^{k+1}l})^l\psi_{q}(\tau\phi_{2^{k+1}l})^{-l}=\tau^l\psi_{q}\tau^{-l}=\psi_{q}^{-1}.
\end{equation*}
Finally, since $2^{k+1}$ is even
\begin{equation*}
 (\tau\phi_{2^{k+1}l})^{2^{k+1}}\psi_{q}(\tau\phi_{2^{k+1}l})^{-2^{k+1}}=\tau^{2^{k+1}}\psi_{q}\tau^{-2^{k+1}}=\psi_{q},
\end{equation*}
which means that $\Psi(z)=(\tau\phi_{2^{k+1}l})^{2^{k+1}}$ and $\Psi(y)=\psi_{q}$ commute.

\eqref{it:T.B} We claim that $\gcd(8\cdot3^k,l)=1$, suppose $t$ divides both $3^k\cdot8$ and $l$, since $l$ is odd, $t$ does not divide $8$ and must divide $3^k$, but $\gcd(3,l)=1$ so $t=1$.
Since $\phi_{6m}$ is a diagonal matrix we have the commutators
\begin{equation}\label{eq:comm.T}
 [\phi_{6m},\eta]=1,\quad [\phi_{6m},\psi_4]=1,\quad [\phi_{6m},\tau]=1.
\end{equation}
It is also straightforward to check that
\begin{equation}\label{eq:conj.T}
 \eta\psi_4\eta^{-1}=\tau,\quad \eta\tau\eta^{-1}=\psi_4\tau=\left(\begin{smallmatrix}
                                                                    0 & -1\\ 1 & 0
                                                                   \end{smallmatrix}\right),
\end{equation}
and since $\tau^{-1}=-\tau$, $\psi_4^{-1}=-\psi_4$ and $\psi_4\tau=-\tau\psi_4$, it is easy to see that
\begin{equation}\label{eq:conj.T2}
 \eta^2\tau^{-1}\eta^{-2}=\psi_4^{-1},\quad \eta^2\psi_4^{-1}\eta^{-2}=\tau^{-1}\psi_4^{-1}.
\end{equation}
Since $\eta$ and $\phi_{6m}$ commute and $\eta$ has order $3$, $\eta\phi_{6m}$ has order $6m=2\cdot 3^k\cdot l$ and $(\eta\phi_{6m})^{3m}=-\mathbb{I}$:
\begin{equation}\label{eq:ele.ord2}
(\eta\phi_{6m})^{3^k\cdot l}=(\eta\phi_{6m})^{3m}=\eta^{3m}\phi_{6m}^{3m}=-\mathbb{I}.
\end{equation}
We also have that $\eta\phi_{6m}=\tau^2(\eta\phi_{6m})^{2\cdot\frac{3m+1}{2}}$, hence
\begin{equation*}
\begin{split}
 \mathbb{T}_m&=\langle\psi_4,\tau,\eta\phi_{6m}\rangle=\langle\psi_4,\tau,(\eta\phi_{6m})^2\rangle\\
 &=\{\psi_4^\alpha\tau^\beta(\eta\phi_{6m})^{2\gamma}\mid \alpha=0,1; \beta=0,\dots,3;\gamma=0,\dots,3m-1\}.
\end{split}
\end{equation*}
By \eqref{it:ffi.Pk} we have the presentation of the direct product
\begin{multline}\label{eq:prod.pres.T}
P'_{8\cdot 3^k}\times\mathsf{C}_l=\\\langle x,y,z,w\mid x^2=(xy)^2=y^2,zxz^{-1}=y, zyz^{-1}=xy, z^{3^k}=1, w^l=1, w\rightleftarrows xyz\rangle
\end{multline}
where $w\rightleftarrows xyz$ means that $w$ commutes with $x$, $y$ and $z$.
Since $\gcd(3,l)=1$ we have two cases.

\paragraph{\textbf{Case 1:} $l\equiv 2\mod 3$} In this case the isomorphism is given by
\begin{gather*}
\Phi\colon P'_{8\cdot 3^k}\times\mathsf{C}_l\to \mathbb{T}_m\\
x\mapsto\psi_4,\quad y\mapsto\tau,\quad z\mapsto(\eta\phi_{6m})^{2l},\quad w\mapsto(\eta\phi_{6m})^{2\cdot 3^k},
\end{gather*}
which is well-defined since  the images $\Phi(x)$, $\Phi(y)$, $\Phi(z)$ and $\Phi(w)$ satisfy the relations of $P'_{8\cdot 3^k}\times\mathsf{C}_l$ given in \eqref{eq:prod.pres.T}:
by the definitions of $\psi_4$ and $\tau$ and by \eqref{eq:conj.T} we have
\begin{equation*}
 \psi_4^2=\left(\begin{smallmatrix}
                                                                    0 & -1\\ 1 & 0
                                                                   \end{smallmatrix}\right)^2=\tau^2=-\mathbb{I}.
\end{equation*}
By \eqref{eq:ele.ord2} $\Phi(z)=(\eta\phi_{6m})^{2l}$ has order $3^k$ and $\Phi(w)=(\eta\phi_{6m})^{2\cdot 3^k}$ has order $l$ and obviously they commute.
By \eqref{eq:comm.T} and since $\eta$ has order $3$ we also have that
\begin{align*}
(\eta\phi_{6m})^{2\cdot 3^k} \psi_4(\eta\phi_{6m})^{-2\cdot 3^k}&=\eta^{2\cdot 3^k}\psi_4\eta^{-2\cdot 3^k}=\psi_4,\\
(\eta\phi_{6m})^{2\cdot 3^k} \tau(\eta\phi_{6m})^{-2\cdot 3^k}&=\eta^{2\cdot 3^k}\tau\eta^{-2\cdot 3^k}=\tau.
\end{align*}
Since $l\equiv 2\mod 3$, then $2l\equiv 1\mod 3$ and by \eqref{eq:comm.T} and \eqref{eq:conj.T}
\begin{align*}
 (\eta\phi_{6m})^{2l}\psi_4(\eta\phi_{6m})^{-2l}&=\eta^{2l}\psi_4\eta^{-2l}=\eta\psi_4\eta^{-1}=\tau,\\
 (\eta\phi_{6m})^{2l}\tau(\eta\phi_{6m})^{-2l}&=\eta^{2l}\tau\eta^{-2l}=\eta\tau\eta^{-1}=\psi_4\tau.
\end{align*}

\paragraph{\textbf{Case 2:} $l\equiv 1\mod 3$} In this case the isomorphism is given by
\begin{gather*}
\Phi\colon P'_{8\cdot 3^k}\times\mathsf{C}_l\to \mathbb{T}_m\\
x\mapsto\tau^{-1},\quad y\mapsto\psi_4^{-1},\quad z\mapsto(\eta\phi_{6m})^{2l},\quad w\mapsto(\eta\phi_{6m})^{2\cdot 3^k},
\end{gather*}
which is well-defined since  the images $\Phi(x)$, $\Phi(y)$, $\Phi(z)$ and $\Phi(w)$ satisfy the relations of $P'_{8\cdot 3^k}\times\mathsf{C}_l$ given in \eqref{eq:prod.pres.T}
in an analogous way to Case~1.
\end{proof}

\begin{corollary}\label{cor:FI=SF}
The non-trivial finite subgroups of $\SO{4}$ acting freely and isometrically on $\mathbb{S}^3$ coincide with the small finite subgroups of $\U$.
\end{corollary}

Corollary~\ref{cor:FI=SF} has Theorem~\ref{Thm1} as a corollary.

\section{Irreducible representations of small finite subgroups of $\GLn{2}$}\label{sec:irr.rep}

Let $\Gamma$ be a finite cyclic group or a group of the families \eqref{it:ffi.Q}-\eqref{it:ffi.Pk} in Theorem~\ref{thm:fsgfi}.
In this section we are interested on the irreducible representations of $\Gamma$. 
Recall that the \textit{character} $\chi_\rho$ of a representation $\rho\colon\Gamma\to\GLn{n}$ is the function $\chi_\rho\colon\Gamma\to\C$ given by $\chi_\rho(g)=\mathrm{Tr}(\rho(g))$,
where $\mathrm{Tr}(\rho(g))$ is the trace of $\rho(g)$.

As usual, denote by $[\Gamma,\Gamma]$ the \textit{commutator subgroup} of $\Gamma$. Let $\Ab\colon\Gamma\to\Gamma/[\Gamma,\Gamma]$
be the projection homomorphism, and denote the \textit{abelianization} of $\Gamma$ by $\Ab(\Gamma)=\Gamma/[\Gamma,\Gamma]$.
Let $\rho\colon\Gamma\to\GLn{1}$ be a one-dimensional representation. Since $\GLn{1}$ is abelian, by the universal property of the abelianization, the representation $\rho$ factorizes as
\begin{equation}\label{eq:rho.Ab}
\xymatrix{
\Gamma\ar[r]^{\rho}\ar[d]_{\Ab} & \GLn{1}\\
\Ab(\Gamma)\ar[ur]_{\rho_{\Ab}}
}
\end{equation}

Table~\ref{tb:AbG} shows the abelianizations of $\Gamma$ in families \eqref{it:ffi.Q}-\eqref{it:ffi.P120}  in Theorem~\ref{thm:fsgfi}
(binary polyhedral groups) but given by presentation \eqref{eq:2st} (Table~\ref{tbl:BPG}), where $\bar{b}=\Ab(b)$ and $\bar{c}=\Ab(c)$ (see \cite[II\S5-Table~2]{Lamotke:RSIS}).
\begin{table}[h]
\setlength{\extrarowheight}{4pt}
\begin{tabular}{|l|c|l|}\cline{1-3}
$\Gamma< \SU$ & $\mathrm{Ab}(\Gamma)$& Generators\\\hline
$\mathsf{BD}_{2t}$ ($t$ even)&  $\mathsf{C}_2\times\mathsf{C}_2$& generated by $\bar{b}$ and $\bar{c}$ \\\hline
$\mathsf{BD}_{2t}$ ($t$ odd) &  $\mathsf{C}_4$& generated by $\bar{b}$ and $\bar{c}=\bar{b}^2$ \\\hline
$\mathsf{BT}$ & $\mathsf{C}_3$& generated by $\bar{c}$ and $\bar{b}=\bar{c}^{-1}$ \\\hline
$\mathsf{BO}$ & $\mathsf{C}_2$& generated by $\bar{c}$ and $\bar{b}=1$\\\hline
$\mathsf{BI}$ & $\{1\}$&\\
\hline
\end{tabular}
\caption{Abelianizations of binary polyhedral groups.}\label{tb:AbG}
\end{table}

The next lemma gives the abelianization of families \eqref{it:ffi.Dmk} and \eqref{it:ffi.Pk} in Theorem~\ref{thm:fsgfi}.

\begin{lemma}\label{lem:Ab}
Consider the groups $D_{2^{k+1}(2r+1)}$ and $P'_{8\cdot 3^k}$ with the presentations given, respectively, in \eqref{it:ffi.Dmk} and \eqref{it:ffi.Pk} in Theorem~\ref{thm:fsgfi}. We have that
\begin{enumerate}[(i)]
 \item $\Ab(D_{2^{k+1}(2r+1)})=\mathsf{C}_{2^{k+1}}$, $\Ab(x)$ is the generator of $\mathsf{C}_{2^{k+1}}$ and $\Ab(y)=1$.\label{it.Ab.D}
 \item $\Ab(P'_{8\cdot 3^k})=\mathsf{C}_{3^k}$, $\Ab(z)$ is the generator of $\mathsf{C}_{3^k}$ and $\Ab(x)=\Ab(y)=1$.\label{it:Ab.P}
\end{enumerate}
\end{lemma}

\begin{proof}
\eqref{it.Ab.D} If the generators commute, from the third relation in \eqref{it:ffi.Dmk} we get $y^2=1$ and with $y^{2r+1}=1$ implies that $y=1$. The first relation remains the same.

\eqref{it:Ab.P} If the generators commute, from the second and third relations in \eqref{it:ffi.Pk} we have $x=y^{-1}$ and $y=xy$, from which we get $y=y^{-1}y=1$, and thus
 $x=y^{-1}=1$. With $x=y=1$ the first relation is trivially satisfied $x^2=x^2y^2=y^2$. The last relation remains the same.
\end{proof}

\begin{remark}\label{rem:Ab.G}
Since the abelianization of a direct product of groups is the direct product of the abelianizations,
by Table~\ref{tb:AbG} and Lemma~\ref{lem:Ab} the abelianization $\Ab(\Gamma)$ of the groups $\Gamma$ listed in Theorem~\ref{thm:fsgfi} are as follows:
\begin{enumerate}[(A)]
 \item $\Ab(\Gamma)$  is a finite cyclic group, if $\Gamma$ is $\mathsf{BD}_{2t}$ with $t$ odd, $\mathsf{BT}$, $\mathsf{BO}$, $\mathsf{BI}$, $D_{2^{k+1}(2r+1)}$ or $P'_{8\cdot 3^k}$.\label{it:cyc}
 \item $\Ab(\Gamma)$  is the direct product of two finite cyclic groups, if $\Gamma$ is $\mathsf{BD}_{2t}$ with $t$ even,
 or the direct product of a group in \eqref{it:cyc} with a finite cyclic group of relatively prime order.
 \item $\Ab(\Gamma)$  is the direct product of three finite cyclic groups, if $\Gamma$ is the direct product of $\mathsf{BD}_{2t}$ with $t$ even with a finite cyclic group of relatively prime order.
\end{enumerate}
\end{remark}

\subsection{The groups $\mathsf{BT}$, $\mathsf{BO}$ and $\mathsf{BI}$}\label{ssec:reps.binary}

If $\Gamma$ is a finite subgroup of $\SU$ (Table~\ref{tbl:BPG}), its irreducible representions are given in \cite{arciniegaetal:CCSC} where the presentation \eqref{eq:2st} is used.

\paragraph{\textbf{The group $\mathsf{BT}$}} It has $7$ irreducible representations, denoted by $\alpha_j$, $j=1,\dots,7$,
three $1$-dimensional, three $2$-dimensional and one $3$-dimensional.

\paragraph{\textbf{The group $\mathsf{BO}$}} It has $8$ irreducible representations, denoted by $\alpha_j$, $j=1,\dots,8$,
two $1$-dimensional, three $2$-dimensional, two $3$-dimensional and one $4$-dimensional.

\paragraph{\textbf{The group $\mathsf{BI}$}} It has $9$ irreducible representations, denoted by $\alpha_j$, $j=1,\dots,9$,
One $1$-dimensional, two $2$-dimensional, two $3$-dimensional, two $4$-dimensional, one $5$-dimensional and one $6$-dimensional.

\subsection{The group $\mathsf{BD}_{2q}$}\label{ssec:BD2q}

It has $q+3$ irreducible representations. It has $4$ one-dimensional representations, denoted by $\alpha_j$, $0\leq j\leq 3$, they correspond to the irreducible representations of its abelianization.
From Table~\ref{tb:AbG}, for $q$ even we have that $\Ab(\mathsf{BD}_{2q})\cong\mathsf{C}_2\times\mathsf{C}_2$  and the four one-dimensional irreducible representations correspond to the tensor product of the irreducible representations of $\mathsf{C}_2$ with themselves.
For $q$ odd $\Ab(\mathsf{BD}_{2q})\cong\mathsf{C}_4$, in this case the four one-dimensional irreducible representations correspond to the irreducible representations of $\mathsf{C}_4$.
 \begin{align*}
\text{$r$ \textbf{even:}}\qquad\alpha_0(b)&=\alpha_0(c)=1, &\alpha_1(b)&=-1,\alpha_1(c)=1,\\
\alpha_2(b)&=1, \alpha_2(c)=-1,& \alpha_3(b)&=\alpha_3(c)=-1.\\[15pt]
\text{$r$ \textbf{odd:}}\qquad\alpha_0(b)&=\alpha_0(c)=1, &\alpha_1(b)&=i, \alpha_1(c)=-1,\\
\alpha_2(b)&=-1,\alpha_2(c)=1,& \alpha_3(b)&=-i,\phi_3(c)=-1.
\end{align*}
It also has $q-1$ two-dimensional representations, denoted by $\rho_t$, $1\leq t\leq q-1$, given for any $r$ by:
\begin{equation}\label{eq:2d.rep.BD}
\rho_t(b)=\left(\begin{smallmatrix}0&1\\(-1)^t&0\end{smallmatrix}\right),\qquad \rho_t(c)=\left(\begin{smallmatrix}\zeta_{2q}^{t}&0\\0&\zeta_{2q}^{-t}\end{smallmatrix}\right),
\end{equation}
where $\rho_1$ is the natural representation.

\subsection{The group $D_{2^{k+1}(2r+1)}$}\label{ssec:D2kq}
It has $2^{k}(r+2)$ irreducible representations\footnote{For $k=1$, $D_{4(2r+1)}\cong\mathsf{BD}_{2(2r+1)}$. In this case $q=2r+1$ and $q+3=2r+4=2(r+2).$}.
It has $2^{k+1}$ one-dimensional representations, denoted by $\alpha_j$, $0\leq j\leq 2^{k+1}-1$, they correspond to the irreducible representations of its abelianization, by Lemma~\ref{lem:Ab} $\Ab(D_{2^{k+1}(2r+1)})=\mathsf{C}_{2^{k+1}}$ and they are given by
\begin{equation}\label{eq:D.1dim}
\alpha_j(x)=\zeta_{2^{k+1}}^{j},\quad \alpha_j(y)=1.
\end{equation}
Consider the two-dimensional representations
\begin{equation}\label{eq:D2kq.2dir}
\varrho_{t,s}(x)=\zeta_{2^{k+1}}^s\left(\begin{smallmatrix}0&1\\(-1)^t&0\end{smallmatrix}\right),\qquad \varrho_{t,s}(y)=\left(\begin{smallmatrix}\zeta_{2r+1}^{t}&0\\0&\zeta_{2r+1}^{-t}\end{smallmatrix}\right).
\end{equation}
It is straightforward to check that $\varrho_{t,s}$ satisfies the relations in \eqref{it:ffi.Dmk} so it is indeed a representation of $D_{2^{k+1}(2r+1)}$.
An easy computation shows that the inner product of the character of the representation $\varrho_{t,s}$ with itself is equal to $1$, which shows that $\varrho_{t,s}$ is irreducible.
Naturally we have that $t=1,\dots,2r$ and $s=0,\dots,2^{k+1}-1$, but comparing characters one can check that
\begin{equation}\label{eq:D2kq.2d=}
\varrho_{t,s}=\varrho_{t,2^{k}+s},\quad\text{and}\quad \varrho_{2r+1-t,s}=\varrho_{t,2^{k-1}+s},
\end{equation}
therefore the list of different two-dimensional representations is $\varrho_{t,s}$ with $t=1,\dots,2r$ and $s=0,\dots,2^{k-1}-1$
(or equivalently $t=1,\dots,r$ and $s=0,\dots,2^{k}-1$), so they are $2^{k}r$ of them.
They are all different since they have different characters.
This is the complete list of irreducible representations of $D_{2^{k+1}(2k+1)}$.
We have listed $2^{k+1}$ one-dimensional representations and $2^{k}r$ two-dimensional representations. The sum of the squares of their dimensions is
\begin{equation*}
    2^{k+1} + 2^2 2^{k}r = 2^{k+1} + 2^{k+2}r=2^{k+1}(2r+1),
\end{equation*}
which is the order of $D_{2^{k+1}(2k+1)}$. By~\cite[Chapter~9, Theorem~5.9]{zbMATH00425998}, there can not exist any other irreducible representation.

\begin{remark}\label{rem:k=2}
When $k=1$ we have $s=0$ in $\varrho_{t,s}$ and we recover the $2$-dimensional irreducible representations of $\mathsf{BD}_{2(2r+1)}$ given in \cite{arciniegaetal:CCSC} via the isomorphism given in \eqref{eq:Dmk.BDk}.
\end{remark}

\subsection{The group $P'_{8\cdot 3^k}$}\label{ssec:Pp83k}

It has $7\cdot 3^{k-1}$ irreducible representations\footnote{For $k=1$ this agrees with the fact that $P'_{24}\cong\mathsf{BT}$.}.
It has $3^k$ one-dimensional representations, denoted by $\alpha_j$, $0\leq j\leq 3^k-1$, they correspond to the irreducible representations of its abelianization, by Lemma~\ref{lem:Ab} $\Ab(P'_{8\cdot 3^k})=\mathsf{C}_{3^k}$ and they are given by
\begin{equation*}
\alpha_j(z)=\zeta_{3^k}^{j},\quad \alpha_j(x)=\alpha_j(y)=1.
\end{equation*}
It also has $3^k$ two-dimensional representations, denoted by $\varrho_{s}$ with $s=0,\dots,3^{k}-1$, given by
\begin{equation}\label{eq:2drPp83k}
\begin{aligned}
 \varrho_{s}(x)&=\left(\begin{smallmatrix}
                       0 & \zeta_3^2\\
                       -\zeta_3 & 0
                      \end{smallmatrix}\right), & \varrho_{s}(y)&=\left(\begin{smallmatrix}
                                                                \zeta_3^2 & 1\\
                                                                \zeta_3^2 & -\zeta_3^2
                                                               \end{smallmatrix}\right), & \varrho_{s}(z)&=\zeta_{3^k}^s\left(\begin{smallmatrix}
                                                                                                          0 & \zeta_3\\
                                                                                                          -\zeta_3^2 & -1
                                                                                                         \end{smallmatrix}\right).
\end{aligned}
\end{equation}
It has $3^{k-1}$ three-dimensional irreducible representations $\varsigma_s$, $s=0,\dots,3^{k-1}-1$ given by
\begin{equation}\label{eq:3drPp83k}
\begin{aligned}
\varsigma_s(x)&=\left(\begin{smallmatrix}
                     -1 & -1 & -1\\
                     0 & 0 & 1\\
                     0 & 1 & 0
                    \end{smallmatrix}\right), & \varsigma_s(y)&=\left(\begin{smallmatrix}
                                                                     0 & 0 & 1\\
                                                                     -1 & -1 & -1\\
                                                                     1 & 0 & 0
                                                                    \end{smallmatrix}\right), & \varsigma_s(z)&=\zeta_{3^k}^s\left(\begin{smallmatrix}
                                                                                                                     -1 & -1 & -1\\
                                                                                                                     0 & 1 & 0\\
                                                                                                                     1 & 0 & 0
                                                                                                                    \end{smallmatrix}\right).
\end{aligned}
\end{equation}
It is straightforward to check that $\varrho_{s}$ and $\varsigma_s$ satisfy the relations in \eqref{it:ffi.Pk} so they are indeed representations of $P'_{8\cdot 3^k}$.
An easy computation shows that the inner product of the characters of the representation $\varrho_{s}$ and $\varsigma_s$ with themselves is equal to $1$, which shows that they are irreducible. They are all different since they have different characters.
This is the complete list of irreducible representations of $P'_{8\cdot 3^k}$ since the sum of the square of their dimensions is
\begin{equation*}
 3^k+4(3^k)+9(3^{k-1})=3^k+4(3^k)+3(3^k)=8(3^k),
\end{equation*}
which is the order of $P'_{8\cdot 3^k}$. By~\cite[Chapter~9, Theorem~5.9]{zbMATH00425998}, there can not exist any other irreducible representation.

\begin{remark}
When $k=1$,  $\varrho_{0}$, $\varrho_{3^{k-1}}$ and $\varrho_{2\cdot3^{k-1}}$ are the $3$ two-dimensional 
irreducible representations of $\mathsf{BT}$ and $\varsigma_0$ is the three-dimensional irreducible representation of $\mathsf{BT}$.
They are equivalent to the irreducible representations given in \cite{arciniegaetal:CCSC} via the isomorphism given by the composition of the isomorphism $P'_{24}\cong P_{24}$ given in \eqref{eq:P24.P24} and the isomorphism $P_{24}\cong\mathsf{BT}$ given in Lemma~\ref{lem:P=BP}.
\end{remark}

\section{Chern and Cheeger-Chern-Simons classes and CCS-numbers of $\rho$}\label{sec:CCS}

Our next goal is to classify flat vector bundles over spherical $3$-manifolds. In this section we recall the definitions and some results on flat vector bundles and
their Chern and Cheeger-Chern-Simons classes. For details see \cite{arciniegaetal:CCSC} and its references.

\subsection{Flat vector bundles}

Let $G$ be a Lie group and denote by $G^d$ the group $G$ but with the discrete topology.
Let $\BG[G]$ be the classifying space of $G$. Note that the classifying space $\BG[G^d]$ is an Eilenberg-McLane space of type $K(G^d,1)$.
The continuous map $G^d \to G$ given by the identity induces a map between the classifying spaces $\iota \colon \BG[G^d] \to \BG[G]$.

Let $M$ be a compact smooth manifold.
There is a one-to-one correspondence
\begin{equation}\label{eq:Vect.class.map}
\text{Vect}_{n,\C}(M) \Longleftrightarrow [M, \BG[\GLn{n}]].
\end{equation}
between the set $\text{Vect}_{n,\C}(M)$, of isomorphism classes of complex vector bundles of rank $n$ over $M$,
and the set $[M, \BG[\GLn{n}]]$ of homotopy classes of maps from $M$ to the classifying space $\BG[\GLn{n}]$.
To every vector bundle $\xi$ of rank $n$ over $M$ corresponds the homotopy class of its \textit{classifying map} $f\colon M\to \BG[\GLn{n}]$,
such that $\xi$ is isomorphic to the pull-back by $f$ of the \textit{universal vector bundle} $\gamma^n$ of rank $n$ over $\BG[\GLn{n}]$, that is $\xi=f^*(\gamma^n)$.
There are concrete models for $\BG[\GLn{n}]$ and $\gamma^n$.
Let $\mathrm{G}_n(\C^K)$ be the \textit{Grassmannian} of $n$-planes in $\C^K$, then
$\BG[\GLn{n}]=\mathrm{G}_n=\mathrm{G}_n(\C^\infty)=\varinjlim \mathrm{G}_n(\C^K)$ and the \emph{universal vector bundle} is the canonical complex bundle $\gamma^n$ of rank $n$ over $\mathrm{G}_n$ \cite[Theorem~8-7.2]{Husemoller:Bundles}.

A vector bundle $\xi$ over $M$ is \textit{flat} if it admits a flat connection, that is, a connection with zero curvature (see \cite[\S3]{Labourie:LRSG} for definitions).
This is equivalent to the classifying map $f\colon M\to \BG[\GLn{n}]$ of $\xi$ factorizing through $\BG[G^d]$, that is
\begin{equation}\label{eq:flat.fact}
\xymatrix{
&\BG[\GLn{n}^d] \ar[d]^\iota \\
M \ar[r]^{f} \ar@{.>}[ru]^{\bar{f}} & \BG[\GLn{n}].
}
\end{equation}

In \cite[Theorem~1]{Narasimhan-Ramanan:EUC} Narasimhan and Ramanan proved that the
universal bundle $\gamma^n\to \mathrm{G}_n$ has a \textit{universal connection} $\nabla_\mathrm{univ}$. We denote it by $(\gamma^n,\BG[\GLn{n}],\nabla_\mathrm{univ})$.
The pull-back (with connection) of $\gamma^n$ by the map $\iota \colon \BG[\GLn{n}]^d \to \BG[\GLn{n}]$ is a \emph{universal flat vector bundle} with \textit{universal flat connection}. We denote it by $(\gamma^n_d,\BG[\GLn{n}]^d,\nabla_\mathrm{univ}^d)=(\iota^*\gamma^n,\BG[\GLn{n}]^d,\iota^*\nabla_\mathrm{univ})$.
Thus, any flat rank $n$ vector bundle $\xi$ over $M$ is the pull back of the universal flat vector bundle $\gamma^n_d$ by the map $\bar{f}$.
If we denote by $\text{Vect}_{n,\C}^{\mathrm{Flat}}(M)$ the set of isomorphism classes of flat rank $n$ complex vector bundles over $M$ we have
a one-to-one correspondence
\begin{equation}\label{eq:Flat.Rep}
 \text{Vect}_{n,\C}^{\mathrm{Flat}}(M) \Longleftrightarrow [M, \BG[\GLn{n}^d]] \Longleftrightarrow \Hom(\pi_1(M),\GLn{n}^d).
\end{equation}
In other words, we can classify flat vector bundles over $M$ using invariants of representations of the fundamental group of $M$ and \textit{vice versa}.

\subsection{Chern and Cheeger-Chern-Simons classes of a representation}\label{ssec:CCScl}

Let $M$ be a compact smooth manifold and $\rho\colon\pi_1(M)\to\GLn{n}^d$ a representation of its fundamental group. By \eqref{eq:Flat.Rep} to $\rho$ corresponds a flat vector bundle $V_\rho$ over $M$ which is given by $V_\rho=\tilde{M}\times_\rho\C^n\to M$, where $\tilde{M}$ is the universal cover of $M$ and $\tilde{M}\times_\rho\C^n$ is the quotient of $\tilde{M}\times\C^n$ by the action of $\pi_1(M)$,
given on the first factor by the canonical action of $\pi_1(M)$ on $\tilde{M}$ and via the representation $\rho$ on the second factor.

The $k$-th \textit{Chern class} of $\rho$ is the $k$-th Chern class $c_{\rho,k}$ of the flat vector bundle $V_\rho$
\begin{equation*}
 c_{\rho,k}=c_k(V_\rho)\in H^{2k}(M;\Z).
\end{equation*}

Cheeger and Simons \cite{Cheeger-Simons:DCGI} defined a group of \textit{differential characters} (\textit{differential cohomology} \cite{Bar-Becker:DiffChar}) $\widehat{H}^k(\BG[\GLn{n}];\C/\Z)$ containing the group $H^{k-1}(\BG[\GLn{n}];\C/\Z)$. Using the universal bundle with connection $(\gamma^n,\BG[\GLn{n}],\nabla_\mathrm{univ})$
they constructed in $\widehat{H}^k(\BG[\GLn{n}];\C/\Z)$ a \textit{universal Chern differential character} $\widehat{C}_k$.

Since the pull-back $(\gamma^n_d,\BG[\GLn{n}]^d,\nabla_\mathrm{univ}^d)$ of $(\gamma^n,\BG[\GLn{n}],\nabla_\mathrm{univ})$
by the map $\iota \colon \BG[\GLn{n}]^d \to \BG[\GLn{n}]$ is \textit{flat}, the pull-back $\widehat{c}_k$ of the universal Chern differential character $\widehat{C}_k$ by $\iota$ actually lies in
$H^{k-1}(\BG[\GLn{n}^d];\C/\Z)$, that is
\begin{equation}\label{eq:UCCSc}
 \widehat{c}_k=\iota^*(\widehat{C}_k)\in H^{2k-1}(\BG[\GLn{n}]^d;\C/\Z),
\end{equation}
this is the \emph{universal $k$-th Cheeger-Chern-Simons class for flat bundles}.
Let $\BG[\rho]\colon M \to \BG[\GLn{n}^d]$ be the map induced between classifying spaces by $\rho$ given by \eqref{eq:Flat.Rep}. The $k$-th \textit{Cheeger-Chern-Simon class} of $\rho$ is given by
\begin{equation*}
\widehat{c}_{\rho,k}=(B\rho)^*(\widehat{c}_k)\in H^{2k-1}(M;\C/\Z).
\end{equation*}
By the Universal Coefficient Theorem and using that $\C/\Z$ is divisible there is an isomorphism
\begin{equation}\label{eq:coh.mor}
H^{2k-1}(M;\C/\Z) \cong \Hom(H_{2k-1}(M;\Z),\C/\Z).
\end{equation}
Thus, the Cheeger-Chern-Simons classes $\widehat{c}_{\rho,k}$ can be identified as homomorphisms $\widehat{c}_{\rho,k} \colon H_{2k-1}(M;\Z) \to \C/\Z$.
Let $\kappa\in H_{2k-1}(M;\Z)$, we call the image $\widehat{c}_{\rho,k}(\kappa)$ the \textit{$k$-th CCS-number of $\rho$ with respect to $\kappa$}.

Cheeger-Chern-Simons classes have the following functorial property \cite[(8.3)]{Cheeger-Simons:DCGI}:
\begin{proposition}\label{prop:ccs-pullback}
Let $M$ and $N$ be compact smooth manifolds and $f\colon N\to M$ a smooth map.
Let $f_*\colon\pi_1(N)\to\pi_1(M)$ and $f^*\colon H^{2k-1}(M;\C/\Z)\to H^{2k-1}(M;\C/\Z)$ be the homomorphisms induced by $f$ on fundamental groups and cohomology respectively.
Let $\rho\colon\pi_1(M)\to\GLn{n}$ be a representation of the fundamental group of $M$ and consider the representation of the fundamental group of $N$ given by
$\rho\circ f_*\colon\pi_1(N)\to\GLn{n}$. Then
\begin{equation*}
 \widehat{c}_{\rho\circ f_*,k}=f^*(\widehat{c}_{\rho,k}).
\end{equation*}
Using isomorphism \eqref{eq:coh.mor}, if $\kappa\in H_{2k-1}(M;\Z)$ then
\begin{equation*}
\widehat{c}_{\rho\circ f_*,k}(\kappa)= \widehat{c}_{\rho,k}(\bar{f}_*(\kappa)),
\end{equation*}
where $\bar{f}_*\colon H_{2k-1}(N;\Z)\to H_{2k-1}(M;\Z)$ is the homomorphism induced by $f$.
\end{proposition}

\begin{remark}\label{rem:G.M.ccs}
One can also define the Chern and Cheeger-Chern-Simons classes of a representation $\rho\colon \Gamma\to\GLn{n}$ as cohomology classes of the classifying space $\BG$.
The \emph{$k$-th Chern class of $\rho$ in $\BG$} is the $k$-th Chern class $c_k(\rho)\in H^{2k}(\BG;\Z) $ of the vector bundle $\EG\times_\rho\C^n$ over $\BG$, where $\EG$ is the universal cover of the classifying space $\BG$.
Let $\widehat{c}_k\in H^{2k-1}(\BG[\GLn{n}]^d;\C/\Z)$ be the $k$-th universal Cheeger-Chern-Simons class for flat bundles \eqref{eq:UCCSc}.
Let $\BG[\rho] \colon \BG \to \BG[\GLn{n}]^d$ be the map induced between classifying spaces by $\rho$.
The \emph{$k$-th Cheeger-Chern-Simons class of $\rho$ in $\BG$} is the pullback
\begin{equation}\label{eq:ccs.BG}
\widehat{c}_k(\rho)=\BG[\rho]^*(\widehat{c}_k)\in H^{2k-1}(\BG;\C/\Z).
\end{equation}
	
Let $M$ be a smooth manifold with fundamental group $\pi_1(M)=\Gamma$ and consider the map $\phi\colon M\to \BG$.
Then the $k$-th Cheeger-Chern-Simons class is given by the pull-back $\widehat{c}_{\rho,k}=\phi^*(\widehat{c}_k(\rho))$.
\end{remark}

\subsection{CCS-numbers of closed oriented $3$-manifolds}

Let $L$ be a closed oriented $3$-manifold and $\rho\colon\pi_1(L)\to\GLn{n}$ a representation of its fundamental group. There are only two non-zero Cheeger-Chern-Simons classes of $\rho$, namely
$\widehat{c}_{\rho,1}\in H^{1}(L;\C/\Z)$ and $\widehat{c}_{\rho,2}\in H^{3}(L;\C/\Z)$. As before, if $\nu\in H_{1}(L;\Z)$ we have the \textit{first CCS-number of $\rho$ with respect to $\nu$} given by $\widehat{c}_{\rho,1}(\nu)$. On the other hand, we always take the \textit{second CCS-number of $\rho$} with respect to the fundamental class $[L]\in H^{3}(L;\Z)$ of $L$, since the second CCS-number with respect to any other class in $H^{3}(L;\Z)$ is a multiple of $\widehat{c}_{\rho,2}([L])$.

A representation $\rho \colon \pi_1(L) \to \GLn{n}$ is said to be \textit{topologically trivial} if the vector bundle $V_\rho\to M$ is topologically trivial, i.\ e., it is isomorphic
as a topological vector bundle to the product bundle $M\times\C^n\to M$. For a topologically trivial representation $\rho$, one can compute its second CCS-number via the reduced $\xi$-invariant of the Dirac operator of $L$ twisted by $\rho$ defined by Atiyah, Patodi and Singer in \cite{Atiyah-Patodi-Singer:SARGIII}.

\subsubsection{The $\tilde{\xi}$-invariant}

Let $L$ be a compact oriented $3$-manifold, then it is a spin-manifold \cite[Chapter~VII, Theorem~1]{Kirby:T4M} and it has a Dirac operator $D$, which is a first order, self-adjoint,
elliptic operator, acting on sections of the spinor bundle $\mathcal{S}\to L$ (see~\cite[\S~3.4]{Jost:RGGA} or \cite[Example~5.9]{Lawson-Michelsohn:SpinGeo}).
The operator $D$ has a real discrete spectrum and one can define a complex valued function $\eta(s;A)=\sum_{\lambda\not=0}(\sign\lambda)\abs{\lambda}^{-s}$ called the $\eta$-series
which converges for $\mathrm{Re}(s)$ sufficiently large and extends to a meromorphic function on the whole complex $s$-plane and it is finite at $s=0$.
The value $\eta(D)=\eta(0;D)$ is the \textit{$\eta$-invariant of $D$}. There is a version which takes into account the zero eigenvalues of $D$ by setting $\xi(D)=\frac{h+\eta(D)}{2}$,
where $h=\dim(\ker(D))$. Given a representation $\rho\colon\pi_1(L)\to\GLn{n}$ one can couple the Dirac operator $D$ to the flat vector bundle $V_\rho\to L$
to get a twisted operator $D_\rho$ acting on sections of the vector bundle $\mathcal{S}\otimes V_\rho$.
The operator $D_\rho$ may no longer be self-adjoint, but it has a self-adjoint symbol, and this is enough to define the $\eta$-function $\eta(s;D_\rho)$
and get its $\xi$-invariant $\xi_\rho(D)=\xi(D_\rho)$ and its \textit{reduced} version
\begin{equation}\label{eq:red.xi}
 \tilde{\xi}_\rho(D):=\xi_\rho(D)-n\xi(D)\in\C/\Z,
\end{equation}
which by \cite[Section~2]{Atiyah-Patodi-Singer:SARGIII} its reduction modulo $\Z$ is a homotopy invariant of $D$.

Let $\rho \colon \pi_1(L) \to \GLn{n}$ be a topologically trivial representation, then \cite[Theorem~3.3, Corollary~3.4]{arciniegaetal:CCSC}
\begin{equation}\label{eq:CCS2=xi}
\widehat{c}_{\rho,2}([L])=\tilde{\xi}_\rho(D).
\end{equation}

\subsubsection{CCS-numbers of rational homology $3$-spheres}

Now we restrict to the case when the $3$-manifold $L$ is a rational homology $3$-sphere.
Let $\rho\colon\pi_1(L)\to \GLn{n}$ be a representation.
By Poincaré duality $H^1(L;\C)=H^2(L;\C)=0$, by the cohomology long exact sequence of $L$ corresponding to the short exact sequence of coefficients $0 \to \Z \to \C \to \C/\Z \to 0$ we have a correspondence between the first Cheeger-Chern-Simons class and the first Chern class of $\rho$ under the isomorphism
\begin{equation}\label{eq:ccs-c1}
\begin{aligned}
H^1(L;\C/\Z) &\cong H^2(L;\Z)\\
\widehat{c}_{\rho,1} &\mapsto c_{\rho,1}.
\end{aligned}
\end{equation}
In this case we have the following properties \cite[\S 3.4]{arciniegaetal:CCSC}.

\begin{theorem}\label{thm:rhs}
Let $L$ be a rational homology $3$-sphere. Let $\rho\colon\pi_1(L)\to \GLn{n}$ be a representation of its fundamental group
and $\det(\rho)=\det\circ\rho$, where $\det\colon \GLn{n} \to \GLn{1}$ is the determinant homomorphism. Then
\begin{enumerate}
 \item The representation $\rho$ is topologically trivial if and only if $c_{\rho,1}=0$, if and only if $\widehat{c}_{\rho,1}=0$.\label{it:top.triv}
 \item $\widehat{c}_{\rho,1}=\widehat{c}_{\det(\rho),1}$.\label{it:c1.r.det}
 \item $\widehat{c}_{\rho,2}([L])=\tilde\xi_{\rho}(D)-\tilde\xi_{\det(\rho)}(D)$.\label{it:xi.non.triv}
\end{enumerate}
\end{theorem}

We can use Theorem~\ref{thm:rhs}-\eqref{it:c1.r.det} to compute the first CCS-numbers of $\rho$.

\begin{proposition}[{\cite[(0.2)]{DUPONT1994}}]\label{prop:1stCssn}
Let $\rho\colon\pi_1(L)\to\GLn{n}$ be a representation and $\det(\rho)\colon\pi_1(L)\to\GLn{1}$ its determinant. Then, the first CCS-number of $\rho$ with respect to
$\bar{\nu}\in H_{1}(L;\Z)\cong\Ab(\pi_1(L))$ is given by
\begin{equation}\label{eq:df}
\widehat{c}_{\rho,1}(\bar{\nu})=\widehat{c}_{\det(\rho),1}(\bar{\nu})=\frac{1}{2\pi i}\log(\det(\rho)_{\mathrm{Ab}}(\bar{\nu})),
\end{equation}
where $\det(\rho)_{\mathrm{Ab}}\colon\Ab(\pi_1(L))\to\GLn{1}$ is defined by diagram \eqref{eq:rho.Ab}.
\end{proposition}

\begin{proposition}[{\cite[Prop.~5.1]{arciniegaetal:CCSC}}]\label{prop:1drnt.xi0}
Let $\rho\colon\pi_1(L)\to\GLn{1}$ be a one-dimensional representation which is not topologically trivial. Then $\widehat{c}_{\rho,2}([L])=0$.
\end{proposition}

\begin{proof}
If $\rho$ is not topologically trivial, by Theorem~\ref{thm:rhs}-\eqref{it:xi.non.triv} the second CCS-number is given by  $\widehat{c}_{\rho,2}([L])=\tilde\xi_{\rho}(D)-\tilde\xi_{\det(\rho)}(D)$,
but since it is $1$-dimensional, we have that $\alpha=\det(\alpha)$ and thus $\widehat{c}_{\alpha,2}([L])=0$.
\end{proof}

\subsubsection{Spherical $3$-manifolds, CCS-numbers and the $\tilde{\xi}$-invariant}\label{ssec:s3m.ccsn.xi}

Let $\Gamma$ be a group listed in Theorem~\ref{thm:fsgfi} and let $\varsigma\colon\Gamma\to U(2)$ be a faithful fixed-point free unitary representation, that is, $\det(I-\varsigma(g))\neq0$
for every $g\in\Gamma$.
Hence, $\Gamma$ acts on $\mathbb{S}^3$ via $\varsigma$ freely and isometrically on $\mathbb{S}^3$ and the quotient manifold $M=\mathbb{S}^3/\varsigma(\Gamma)$ is a spherical $3$-manifold.
The homology groups of $M$ are $H_0(M)=\Z$, $H_1(M)=\Ab(\Gamma)$, $H_2(M)=0$ and $H_3(M)=\Z$ (compare with \cite[p.~45]{Lamotke:RSIS}). Thus, $M$ is a rational homology sphere and
\eqref{eq:ccs-c1}, Theorem~\ref{thm:rhs} and Proposition~\ref{prop:1drnt.xi0} apply.

Let $\lambda_1(g)$ and $\lambda_2(g)$ be the eigenvalues of $\varsigma(g)\in\U$ and define
\begin{equation}\label{eq:def.spin}
 \mathrm{def}\bigl(\varsigma(g)\bigr)=\frac{\sqrt{\lambda_1(g)\lambda_2(g)}}{(\lambda_1(g)-1)(\lambda_2(g)-1)}=\frac{\sqrt{\det(\varsigma(g))}}{1-\mathrm{Tr}(\varsigma(g))+\det(\varsigma(g))},
\end{equation}
where $\mathrm{Tr}(\varsigma(g))$ is the trace of $\varsigma(g)$.
Notice that the denominator is non-zero because
\begin{equation*}
 0\neq\det(I-\varsigma(g))=(\lambda_1(g)-1)(\lambda_2(g)-1),
\end{equation*}
or because by Corollary~\ref{cor:FI=SF}, $\varsigma(\Gamma)$ is a small finite subgroup of $\U$ and therefore its elements $\varsigma(g)\neq I$ have no eigenvalues equal to $1$.

The next theorem gives a formula to compute the reduced $\xi$-invariant of the Dirac operator $D_\rho$ of $M$ twisted by a representation $\rho\colon\Gamma=\pi_1(M)\to\U[k]$
(see \cite[Lemma~2.1]{Gilkey:EIKTODSSF} or \cite[Lemma~4.1.1]{Wang:TKTWCEI}).

\begin{theorem}\label{thm:xi.ssf}
Let $\varsigma\colon\Gamma\to U(2)$ be a fixed-point free representation and let $M=\mathbb{S}^3/\varsigma(\Gamma)$. Let $\rho\colon\Gamma=\pi_1(M)\to\U[k]$ be a representation of $\Gamma$. Then
\begin{equation*}
\tilde{\xi}_\rho(D) =\frac{1}{|\Gamma|}\sum_{\substack{g\in \Gamma,\\g\neq 1}}\bigl(\mathrm{Tr}(\rho(g))-k\bigr) \mathrm{def}\bigl(\varsigma(g)\bigr).
\end{equation*}
\end{theorem}

y\section{Classification of flat vector bundles over spherical $3$-manifolds}\label{sec:Class.FVB}

Let $M=\mathbb{S}^3/\Gamma$ be a spherical $3$-manifold where $\Gamma$ is a group listed in Theorem~\ref{thm:fsgfi}. Let $\rho\colon\pi_1(M)\cong\Gamma\to\GLn{n}$ be an
irreducible representation of the fundamental group of $M$.
Consider the triple $[\dim\rho; \widehat{c}_{\rho,1}; \widehat{c}_{\rho,2}]$ given by the dimension and the first and second Cheeger-Chern-Simons-classes of $\rho$, we call it the \textit{CSS-triple} of the representation $\rho$.
By Remark~\ref{rem:G.M.ccs}, the CCS-triple $[\dim\rho; \widehat{c}_{\rho,1}; \widehat{c}_{\rho,2}]$ is equivalent to the CCS-triple $[\dim\rho; \widehat{c}_{1}(\rho); \widehat{c}_{2}(\rho)]$ in $\BG$.

In this section we prove that the CSS-triple is a complete invariant of the pairs $(M,\rho)$, which by \eqref{eq:Flat.Rep} correspond to indecomposable flat vector bundles over $M$.

\begin{remark}\label{rem:CCSt.vCCSn}
By \eqref{eq:coh.mor} we can see the cohomology classes $\widehat{c}_{\rho,1}$ and $\widehat{c}_{\rho,2}$ as homomorphisms
\begin{equation*}
 \widehat{c}_{\rho,1}\colon H_1(M;\Z)\to\C/\Z,\qquad\text{and}\qquad \widehat{c}_{\rho,2}\colon H_3(M;\Z)\to\C/\Z.
\end{equation*}
Since $H_3(M;\Z)\cong\Z$ is generated by the fundamental class $[M]$, $\widehat{c}_{\rho,2}$ is determined by its value on $[M]$, that is, by the second CCS-number $\widehat{c}_{\rho,2}([M])$.
On the other hand, by Remark~\ref{rem:Ab.G} we have that $H_1(M;\Z)\cong\Ab(\Gamma)$ is the direct product of at most three finite cyclic groups,
say $H_1(M;\Z)\cong\mathsf{C}_{l_1}\times\mathsf{C}_{l_2}\times\mathsf{C}_{l_3}$. Let $\nu_i$ be the generator of $\mathsf{C}_{l_i}$ with $i=1,2,3$, thus $\widehat{c}_{\rho,1}$ is determined by its values
$\widehat{c}_{\rho,1}(\nu_i)$ with $i=1,2,3$. Hence, the CSS-triple $[\dim\rho; \widehat{c}_{\rho,1}; \widehat{c}_{\rho,2}]$ is equivalent to consider the vector
\begin{equation*}
(\dim\rho;\widehat{c}_{\rho,1}(\nu_1),\widehat{c}_{\rho,1}(\nu_2),\widehat{c}_{\rho,1}(\nu_3);\widehat{c}_{\rho,2}([M])),
\end{equation*}
where the number of first CCS-numbers depends on the number of cyclic factors of $H_1(M;\Z)$. We call this vector the \textit{vector of CCS-numbers} of $\rho$.
\end{remark}

\subsection{Classification of irreducible representations of $\mathsf{C}_n$}

Let $\mathsf{C}_n$ denote the cyclic group of order $n$ and let $x$ be a generator.
Let $n$ and $q$ be integers with $\gcd(n,q)=1$. Consider the representation
\begin{align*}
\varsigma\colon\mathsf{C}_n &\to U(2)\\
x&\mapsto \left(\begin{smallmatrix}
                 \zeta_n & 0 \\ 0 & \zeta_n^q
                \end{smallmatrix}\right).
\end{align*}
Notice that the image $\varsigma(\mathsf{C}_n)$ is the cyclic subgroup $\mathbb{C}_{n,q}$ of $\U$ defined in \eqref{eq:cyclic.U2}.
The \textit{lens space} $L(n;q)$ is the spherical $3$-manifold given by the quotient $\mathbb{S}^3/\varsigma(\mathsf{C}_n)$.
It has fundamental group $\pi_1(L(n;q))\cong\mathsf{C}_n$, which has $n$ (one-dimensional) irreducible representations, denoted by $\alpha_j^{(q)}\colon\pi_1(L(n;q))\to\U[1]$,
given by
\begin{equation}\label{eq:alphaj}
\alpha_j^{(q)}(x)=\zeta_n^j,\quad\text{for $0\leq j\leq n-1$, where $\zeta_n=e^{\frac{2\pi i}{n}}$.}
\end{equation}

\begin{remark}
 Note that the index $j$ in $\alpha_j$ is shifted by $-1$ with respect to the notation used in \cite{arciniegaetal:CCSC}. This will simplify some computations.
\end{remark}

In order to compute the first CCS-number of the irreducible representations of $\pi_1(L(n;q))$ we need to consider two cases:

\paragraph{\textbf{Case $L(n;-1)$:}}
This case was done in \cite[Table~2]{arciniegaetal:CCSC} and it corresponds to lens spaces which are links of rational double point singularities.

\paragraph{\textbf{Case $L(n;q)$ with $q  \not \equiv -1 \mod n$:}}
If $q \equiv -1 \mod n$, then $L(n;q) = L(n;-1)$ (see Remark~\ref{rem:Cnp.Cmq}). Since  $\gcd(n,q)=1$, there exists $r \in \Z$ such that $0 < r  < n $ and $qr\equiv 1 \mod n$.
By \cite[Theorem~V]{Olum:MMND}, there exist a map
\begin{equation*}
    \varphi \colon L(n;q) \to L(n;-1),
\end{equation*}
such that $\varphi$ induces the identity homomorphism on fundamental groups, that is,
\begin{equation}\label{eq:Id.fg}
\begin{split}
 \varphi_*=Id\colon \pi_1(L(n;q))\cong\mathsf{C}_n&\to\mathsf{C}_n\cong\pi_1(L(n;-1)),\\
 x&\mapsto x
\end{split}
\end{equation}
with $\deg \varphi \equiv -r \mod n$.

Consider the representation $\alpha_{j}^{(-1)}\circ\varphi_*\colon\pi_1(L(n;q))\to \U[1]$, by \eqref{eq:Id.fg} and \eqref{eq:alphaj}
\begin{equation}\label{eq:RepLenteq}
(\alpha_{j}^{(-1)}\circ\varphi_*)(x)=\alpha_{j}^{(-1)}(x)=\zeta_n^{j}=\alpha_{j}^{(q)}(x),\quad\text{for $0\leq j\leq n-1$.}
\end{equation}
By Proposition~\ref{prop:ccs-pullback} and \eqref{eq:RepLenteq} we have that the CCS-numbers of the representations $\alpha_{j}^{(q)}$ are given by
\begin{equation}\label{eq:ccs-num.q}
\widehat{c}_{\alpha_{j}^{(q)},1}(x)= \widehat{c}_{\alpha_{j}^{(-1)},1}\bigl(\overline{\varphi}_*(x)\bigr),
\end{equation}
where $\overline{\varphi}_*\colon H_1(L(n;q);\Z)\to H_1(L(n;-1);\Z)$ is the homomorphism induced by $\varphi$ in homology, which is also the identity homomorphism $\mathsf{C}_n\to\mathsf{C}_n$
given by $x\mapsto x$.

Using the values of the first CCS-numbers of the previous case \cite[Table~2]{arciniegaetal:CCSC} and \eqref{eq:ccs-num.q} we get the first CCS-numbers of the irreducible representations $\alpha_{j}^{(q)}$ for any $q\in\Z$ (including $q \equiv -1 \mod n$) which we present in Table~\ref{tbl:nq}:
\begin{table}[H]
\begin{equation*}
\setlength{\extrarowheight}{4pt}
\begin{array}{|c|c|c|c|c|c|c|}\hline
\mathsf{C}_n & \alpha_{0}^{(q)} &  \alpha_{1}^{(q)} & \dots  &\alpha_{j}^{(q)} &\dots & \alpha_{n-1}^{(q)} \\\hline
\widehat{c}_{\alpha_{j}^{(q)},1}(x) & 0 & \frac{1}{n} & \dots &\frac{j}{n} &\dots &\frac{n-1}{n} \\[3pt]\hline
\end{array}
\end{equation*}
\caption{First CCS-numbers of irreducible representations of $\pi_1(L(n;q))$.}\label{tbl:nq}
\end{table}

\begin{remark}\label{rem:irr.rep.Cn.class}
Since all the first CCS-numbers in Table~\ref{tbl:nq} are different, one-dimensional representations of $\mathsf{C}_n$ are classified by their first CCS-numbers $\widehat{c}_{\alpha_{j}^{(q)},1}(x)$.
\end{remark}

\begin{corollary}\label{cor:vccs.lens}
One-dimensional representations of $\pi_1(L(n;q))\cong\mathsf{C}_n$ are classified by the vectors of CCS-numbers which are of the form $(1;\widehat{c}_{\alpha_j^{(q)},1}(x);0)$.
\end{corollary}

\begin{proof}
From Table~\ref{tbl:nq} and Theorem~\ref{thm:rhs}-\eqref{it:top.triv} the only topologically trivial representation is the trivial representation $\alpha_0^{(q)}$, by \eqref{eq:CCS2=xi} and \eqref{eq:red.xi}
we have $\widehat{c}_{\alpha_0^{(q)},2}([L])=0$. On the other hand, by Proposition~\ref{prop:1drnt.xi0} we have $\widehat{c}_{\alpha_j^{(q)},2}([L])=0$ for $j=1,\dots,n-1$.
Thus, the vector of CCS-numbers $(\dim\alpha_j^{(q)}; \widehat{c}_{\alpha_j^{(q)},1}(x);\widehat{c}_{\alpha_j^{(q)},2}([M]))$ of the representation $\alpha_j^{(q)}$ is of the form
$(1;\widehat{c}_{\alpha_j^{(q)},1}(x);0)$ and Remark~\ref{rem:irr.rep.Cn.class} says that it is enough the second entry given by the first CCS-number $\widehat{c}_{\alpha_j^{(q)},1}(x)$ to classify them.
\end{proof}

\subsection{Classification of irreducible representations of $\mathsf{BD}_{2q}$}\label{ssec:BD.D}

The binary dihedral groups $\mathsf{BD}_{2q}$ can be seen as subgroups of $\U$ (in fact of $\SU$) and correspond to the groups $\mathbb{D}_{q+1,q}$ in Theorem~\ref{thm:fssu2}.
Consider the spherical $3$-manifold $M=\mathbb{S}^3/\mathbb{D}_{q+1,q}$ given by the quotient of the action of $\mathbb{D}_{q+1,q}$ on $\mathbb{S}^3$. Hence we have $\pi_1(M)\cong\mathbb{D}_{q+1,q}\cong\mathsf{BD}_{2q}$.

From Table~\ref{tb:AbG}, for $q$ even we have that $\Ab(\mathsf{BD}_{2q})\cong\mathsf{C}_2\times\mathsf{C}_2$ and
for $q$ odd $\Ab(\mathsf{BD}_{2q})\cong\mathsf{C}_4$.
In both cases there are $4$ one-dimensional irreducible representations $\alpha_1$, $\alpha_2$, $\alpha_3$ and $\alpha_4$, given in subsection~\ref{ssec:BD2q}, where $\alpha_0$ is the trivial one.
The first CCS-numbers of these representations
were computed in \cite[Table~2]{arciniegaetal:CCSC} using \eqref{eq:df}, we summarize the results in Table~\ref{tb:fcn.bd} where $\bar{b},\bar{c}\in H_1(M;\Z)\cong\Ab(\mathsf{BD}_{2q})$.

\begin{table}[H]
\begin{equation*}
\setlength{\extrarowheight}{3pt}
\begin{array}{|c|c|c|c|c||c|c|c|c|c|}\hline
\text{$q$ {\tiny even}} & \alpha_0 & \alpha_1 & \alpha_2 & \alpha_3 & \text{$q$ {\tiny odd\footnotemark}} & \alpha_0 & \alpha_1 & \alpha_2 & \alpha_3\\\hline
\widehat{c}_{\alpha,1}(\bar{b}) & 0 & \frac{1}{2} & 0 & \frac{1}{2} & \widehat{c}_{\alpha,1}(\bar{b}) & 0 & \frac{1}{4} & \frac{1}{2} & \frac{3}{4} \\[3pt]\hline
\widehat{c}_{\alpha,1}(\bar{c}) & 0 & 0 & \frac{1}{2} & \frac{1}{2} &  & &  & & \\[3pt]\hline
\end{array}
\end{equation*}
\caption{First CCS-numbers of one-dimensional irreducible representations of $\mathsf{BD}_{2q}$.}\label{tb:fcn.bd}
\end{table}
\footnotetext{In this case $\bar{c}=\bar{b}^2$ and we have that $\widehat{c}_{\alpha,1}(\bar{c})=\widehat{c}_{\alpha,1}(\bar{b}^2)=2\ \widehat{c}_{\alpha,1}(\bar{b})\mod\Z$.}

\begin{remark}\label{rem:1dr.vcssn}
From Table~\ref{tb:fcn.bd} the four one-dimensional irreducible representations of $\pi_1(M)\cong\mathsf{BD}_{2q}$ have different first CCS-numbers.
By Theorem~\ref{thm:rhs}-\eqref{it:top.triv} the only topologically trivial representation is the trivial representation $\alpha_0$, and by the same arguments as in Remark~\ref{cor:vccs.lens}
the second CCS-numbers of all of these representations are zero. Thus, the vector of CCS-numbers of one-dimensional representations $\alpha_i$ of the fundamental group of $M$, with $i=1,\dots,4$ is of the form
\begin{align*}
(1;\widehat{c}_{\alpha_i,1}(\bar{b}),\widehat{c}_{\alpha_i,1}(\bar{c});0),&\quad \text{for $q$ even,}\\
(1;\widehat{c}_{\alpha_i,1}(\bar{b});0),&\quad \text{for $q$ odd,}
\end{align*}
and it is enough to consider the first CCS-numbers to classify them.
\end{remark}

\begin{remark}\label{rem:tt.2dr}
From the definition of the $q-1$ two-dimensional representations $\rho_t$, with $1\leq t\leq q-1$, of $\pi_1(M)\cong\mathsf{BD}_{2q}$ given in \eqref{eq:2d.rep.BD} we have that
\begin{equation*}
\det\rho_t=\begin{cases}
            \alpha_0 & \text{for $t$ odd,}\\
            \alpha_2 & \text{for $t$ even.}
           \end{cases}
\end{equation*}
Thus, from Theorem~\ref{thm:rhs}-\eqref{it:c1.r.det} we have that
\begin{equation}\label{eq:fccs.2d.BDq}
 \widehat{c}_{\rho_t,1}=\begin{cases}
                         \widehat{c}_{\alpha_0,1} & \text{for $t$ odd,}\\
                         \widehat{c}_{\alpha_2,1} & \text{for $t$ even.}
                        \end{cases}
\end{equation}
Hence, the first Cheeger-Chern-Simons class is not enough to distinguish the two-dimensional representations $\rho_t$, it is necessary to consider the second Cheeger-Chern-Simons class.
From \eqref{eq:fccs.2d.BDq}, Table~\ref{tb:fcn} and Theorem~\ref{thm:rhs}-\eqref{it:c1.r.det} the two-dimensional representations which are topologically trivial are
$\rho_t$ with $t$ odd. To compute the second CCS-number of $\rho_t$ we use \eqref{eq:CCS2=xi} or Theorem~\ref{thm:rhs}-\ref{it:xi.non.triv} and the formula in Theorem~\ref{thm:xi.ssf}.
\end{remark}
First we need a lemma.

\begin{lemma}\label{lem:TS1}
Let $n\in\mathbb{N}$ and $\zeta_n=e^{\frac{2\pi i}{n}}$. We have the equality
\begin{equation*}
\frac{2-\zeta_n^{tj}-\zeta_n^{-tj}}{2-\zeta_n^{j}-\zeta_n^{-j}}=\sum_{i=0}^{t-1}(t-i)\zeta_n^{ij}+\sum_{l=1}^{t-1}(t-l)\zeta_n^{-lj}.
\end{equation*}
\end{lemma}

\begin{proof}
Taking the product
\begin{equation*}
(2-\zeta_n^{j}-\zeta_n^{-j})\left[\sum_{i=0}^{t-1}(t-i)\zeta_n^{ij}+\sum_{l=1}^{t-1}(t-l)\zeta_n^{-lj}\right]
\end{equation*}
one gets a telescoping sum where the only surviving terms are $2-\zeta_n^{tj}-\zeta_n^{-tj}$.
\end{proof}

Now we can compute $\tilde{\xi}_{\rho_t}(D)$, compare with \cite[p.~224]{Cisneros:EITDOSG}.

\begin{proposition}\label{prop:xi.rhot}
Let $\rho_t$ be the two-dimensional irreducible representation of $\pi_1(M)\cong\mathsf{BD}_{2q}$ given in \eqref{eq:2d.rep.BD}. Then
\begin{equation*}
 \tilde{\xi}_{\rho_t}(D) =\frac{1}{4q}(t^2-2qt-2q).
\end{equation*}
\end{proposition}

\begin{proof}
Using the presentation \eqref{eq:2st}, the elements of $\mathsf{BD}_{2q}$ can be written in the form $b^ic^j$ with $i=0,1$ and $j=0,\dots,2q-1$.
The natural representation is $\rho_1\colon\mathsf{BD}_{2q}\to\SU$, hence $\det \rho_1(g)=1$ for any $g\in\mathsf{BD}_{2q}$.
Therefore by \eqref{eq:def.spin} $\mathrm{def}(\rho_1(g))=\frac{1}{2-\mathrm{Tr}(\rho_1(g))}$. Using Lemma~\ref{lem:TS1} and that $\sum_{j=1}^{2q-1}\zeta_{2q}^{ij}=-1$ we have
\begin{align*}
\tilde{\xi}_{\rho_t}(D) &=\frac{1}{4q}\left[\sum_{j=1}^{2q-1} \bigl(\mathrm{Tr}(\rho_j(c^j))-2\bigr) \mathrm{def}\bigl(\rho_1(c^j)\bigr)
+\sum_{j=0}^{2q-1} \bigl(\mathrm{Tr}(\rho_j(bc^j))-2\bigr) \mathrm{def}\bigl(\rho_1(bc^j)\bigr)\right]\\
&=\frac{1}{4q}\left[-\sum_{j=1}^{2q-1}\frac{2-\zeta_{2q}^{tj}-\zeta_{2q}^{-tj}}{2-\zeta_{2q}^{j}-\zeta_{2q}^{-j}}+\sum_{j=0}^{2q-1}\frac{(-2)}{2}\right]\\
&=\frac{1}{4q}\left[-\sum_{j=1}^{2q-1}\left[\sum_{i=0}^{t-1}(t-i)\zeta_{2q}^{ij}+\sum_{l=1}^{t-1}(t-l)\zeta_{2q}^{-lj}\right]-2q\right]\\
&=\frac{1}{4q}\left[-\sum_{i=0}^{t-1}(t-i)\sum_{j=1}^{2q-1}\zeta_{2q}^{ij}-\sum_{l=1}^{t-1}(t-l)\sum_{j=1}^{2q-1}\zeta_{2q}^{-lj}-2q\right]\\
&=\frac{1}{4q}\left[-t(2q-1)+\sum_{i=1}^{t-1}(t-i)+\sum_{l=1}^{t-1}(t-l)-2q\right]\\
&=\frac{1}{4q}\left[-2tq+t+t(t-1)-2q\right]=\frac{1}{4q}(t^2-2qt-2q).
\end{align*}
This proves the proposition.
\end{proof}

\begin{proposition}\label{prop:2dir.vcssn.cls}
Two-dimensional irreducible representations of $\pi_1(M)\cong\mathsf{BD}_{2q}$ are classified by their vector of CCS-numbers.
\end{proposition}

\begin{proof}
Consider $t_1\neq t_2$ with $1\leq t_i\leq q-1$ and $i=1,2$. 
We need to prove that if the two-dimensional irreducible representations $\rho_{t_1}$ and $\rho_{t_2}$ of $\pi_1(M)\cong\mathsf{BD}_{2q}$ have the same vector of CCS-numbers then $\rho_{t_1}=\rho_{t_2}$.
The other implication is trivial.

Suppose that $\rho_{t_1}$ and $\rho_{t_2}$ have the same vector of CCS-numbers, that is, they have the same first and second CCS-numbers.
Thus, they have the same first Cheeger-Chern-Simons class and by Theorem~\ref{thm:rhs}-\eqref{it:c1.r.det}
\begin{equation*}
\widehat{c}_{\det(\rho_{t_1}),1}=\widehat{c}_{\rho_{t_1},1}=\widehat{c}_{\rho_{t_2},1}=\widehat{c}_{\det(\rho_{t_2}),1}.
\end{equation*}
Hence, by Remarks~\ref{rem:tt.2dr} and \ref{rem:1dr.vcssn} 
\begin{equation}\label{eq:same.det}
\det(\rho_{t_1})=\det(\rho_{t_2}),
\end{equation}
and therefore both $t_1$ and $t_2$ are even or odd numbers.
If $t_1$ and $t_2$ are odd, then $\rho_{t_1}$ and $\rho_{t_2}$ are topologically trivial representations and in this case by \eqref{eq:CCS2=xi} we have
\begin{equation*}
\tilde{\xi}_{\rho_{t_1}}(D)=\widehat{c}_{\rho_{t_1},2}([M])= \widehat{c}_{\rho_{t_2},2}([M])=\tilde{\xi}_{\rho_{t_2}}(D).
\end{equation*}
If $t_1$ and $t_2$ are even, by Theorem~\ref{thm:rhs}-\eqref{it:xi.non.triv} we have that
\begin{equation}
\tilde{\xi}_{\rho_{t_1}}(D)-\tilde\xi_{\det(\rho_{t_1})}(D)=\widehat{c}_{\rho_{t_1},2}([M])= \widehat{c}_{\rho_{t_2},2}([M])=\tilde{\xi}_{\rho_{t_2}}(D)-\tilde\xi_{\det(\rho_{t_2})}(D).
\end{equation}
By \eqref{eq:same.det} we have that $\tilde\xi_{\det(\rho_{t_1})}(D)=\tilde\xi_{\det(\rho_{t_2})}(D)$, hence $\tilde{\xi}_{\rho_{t_1}}(D)=\tilde{\xi}_{\rho_{t_2}}(D)$.
In both cases we need to prove that if $\tilde{\xi}_{\rho_{t_1}}(D)=\tilde{\xi}_{\rho_{t_2}}(D)$ then $t_1=t_2$.

Suppose $t_1\neq t_2$ and $\tilde{\xi}_{\rho_{t_1}}(D)=\tilde{\xi}_{\rho_{t_2}}(D)$, then by Proposition~\ref{prop:xi.rhot} we have that
\begin{gather*}
t_1^2-2qt_1-2q=t_2^2-2qt_2-2q,\\
t_1^2-2qt_1=t_2^2-2qt_2,\\
(t_1^2-t_2^2)-2q(t_1-t_2)=0,\\
(t_1-t_2)(t_1+t_2-2q)=0.
\end{gather*}
Since $t_1-t_2\neq0$ then $t_1+t_2=2q$, but $1\leq t_1,t_2\leq q-1$, which is a contradiction. Therefore $t_1=t_2$ and $\rho_{t_1}=\rho_{t_2}$.
\end{proof}

\subsection{Classification of irreducible representations of $D_{2^{k+1}(2r+1)}$}\label{ssec:D2k}

Let $q$ be an odd number. The group $D_{2^{k+1}\cdot q}$ can be seen as a subgroup of $\U$ and by Theorem~\ref{thm:D.T} it corresponds to the group $\mathbb{D}_{2^{k-1}+q,q}$ in Theorem~\ref{thm:fssu2}.
Consider the spherical $3$-manifold $M=\mathbb{S}^3/\mathbb{D}_{2^{k-1}+q,q}$ given by the quotient of the action of $\mathbb{D}_{2^{k-1}+q,q}$ on $\mathbb{S}^3$. Hence we have $\pi_1(M)\cong\mathbb{D}_{2^{k-1}+q,q}\cong D_{2^{k+1}\cdot q}$.

\begin{remark}\label{rem:1dm.D2kq}
From Subsection~\ref{ssec:D2kq} the group $D_{2^{k+1}\cdot q}$ has $2^{k+1}$ one-dimensional irreducible representations, which correspond to the irreducible representations of its abelianization $\mathsf{C}_{2^{k+1}}$; thus, the first CCS-numbers of one-dimensional representations are given as in Table~\ref{tbl:nq} with $n=2^{k+1}$.
So they are classified by the first CCS-number and their vector of CCS-numbers is of the form $(1;\widehat{c}_{\alpha_j,1}(x);0)$.
\end{remark}

On the other hand, from Subsection~\ref{ssec:D2kq}, the group $D_{2^{k+1}(2r+1)}$ has $2^{k}r$ two-dimensional irreducible representations, denoted by $\varrho_{t,s}$, $t=1,\dots,2r$ and $s=0,\dots,2^{k-1}-1$.

\begin{proposition}\label{prop:varrhots.det}
Let $\varrho_{t,s}$ be a two-dimensional irreducible representation of $D_{2^{k+1}(2r+1)}$. Then
\begin{equation*}
 \det(\varrho_{t,s})=\begin{cases}
                      \alpha_{2s}, & \text{if $t$ is odd,}\\
                      \alpha_{2^k+2s}, & \text{if $t$ is even,}
                     \end{cases}
\end{equation*}
where $\alpha_j$ is the one-dimensional irreducible representation given in \eqref{eq:D.1dim}.
\end{proposition}

\begin{proof}
From the definition of $\varrho_{t,s}$ given in \eqref{eq:D2kq.2dir} we have
\begin{align*}
\det(\varrho_{t,s})(x)&=(-1)^{t+1}\zeta_{2^{k+1}}^{2s}=\zeta_{2^{k+1}}^{2^k(t+1)+2s}=\begin{cases}
                                                                                     \zeta_{2^{k+1}}^{2s}=\alpha_{2s}(x), & \text{if $t$ is odd,}\\
                                                                                     \zeta_{2^{k+1}}^{2^k+2s}=\alpha_{2^k+2s}(x), & \text{if $t$ is even.}
                                                                                    \end{cases}\\
\det(\varrho_{t,s})(y)&=1=\alpha_{2s}(y)=  \alpha_{2^k+2s}(y).\qedhere
\end{align*}
\end{proof}

\begin{proposition}\label{prop:fccs.Dkq}
The first CCS-number of the two-dimensional irreducible representation $\varrho_{t,s}\colon D_{2^{k+1}(2r+1)}\to\GLn{2}$ is given by
\begin{equation*}
\widehat{c}_{\varrho_{t,s},1}(x)=\begin{cases}
	\frac{s}{2^{k}} \mod\Z, & \text{for $t$ odd,}\\[5pt]
	\frac{2^{k-1}+s}{2^{k}} \mod\Z, & \text{for $t$ even,}
\end{cases}
\end{equation*}
where $x\in H_1(M;\Z)\cong\Ab(D_{2^{k+1}(2r+1)})\cong\mathsf{C}_{2^{k+1}}$ is the generator.
\end{proposition}

\begin{proof}
From Theorem~\ref{thm:rhs}-\eqref{it:c1.r.det}, Proposition~\ref{prop:varrhots.det} and Table~\ref{tbl:nq} we have that
\begin{equation}
\widehat{c}_{\varrho_{t,s},1}(x)=\widehat{c}_{\det(\varrho_{t,s}),1}(x)=\begin{cases}
                                                                         \widehat{c}_{\alpha_{2s},1}(x)=\frac{2s}{2^{k+1}}=\frac{s}{2^k} & \text{if $t$ is odd,}\\
                                                                         \widehat{c}_{\alpha_{2^k+2s},1}(x)=\frac{2^k+2s}{2^{k+1}}=\frac{2^{k-1}+s}{2^k} &\text{if $t$ is even.}
                                                                        \end{cases}
\end{equation}
We can also compute it using Proposition~\ref{prop:1stCssn}. 
\end{proof}

\begin{remark}
From \eqref{eq:D2kq.2d=} we have that
\begin{equation}
\widehat{c}_{\varrho_{t,s},1}=\widehat{c}_{\varrho_{t,2^{k}+s},1},\quad\text{and}\quad \widehat{c}_{\varrho_{2r+1-t,s},1}=\widehat{c}_{\varrho_{t,2^{k-1}+s},1},
\end{equation}
even though in the second case the first parameter changes parity.
\end{remark}

\begin{remark}
By Remarks~\ref{rem:isos} and \ref{rem:k=2}, for $k=1$ we have $D_{4(2r+1)}\cong\mathsf{BD}_{2(2r+1)}$. In this case $s=0$ and $\varrho_{t,0}=\rho_t$,
where $\varrho_{t,0}$ and $\rho_t$ are, respectively, the two-dimensional representations of $ D_{4(2r+1)}$ and $\mathsf{BD}_{2(2r+1)}$ and the values
of the first CCS-numbers coincide: $\widehat{c}_{\varrho_{t,0},1}(x)=\widehat{c}_{\rho_{t},1}(x)=0$ for $t$ odd, and $\widehat{c}_{\varrho_{t,0},1}(x)=\widehat{c}_{\rho_{t},1}(x)=\frac{1}{2}$ for $t$ even.
\end{remark}

\begin{corollary}\label{cor:fccsn=s}
Let $\varrho_{t_1,s_1}$ and $\varrho_{t_2,s_2}$ be two $2$-dimensional irreducible representations of $D_{2^{k+1}(2r+1)}$ with $1\leq t_1,t_2\leq 2r$ and $0\leq s_1,s_2\leq 2^{k-1}-1$.
If $\widehat{c}_{\varrho_{t_1,s_1},1}=\widehat{c}_{\varrho_{t_2,s_2},1}$, then $s_1=s_2$ and $t_1\equiv t_2\mod 2$.
\end{corollary}

\begin{proof}
First note that given any $2$-dimensional irreducible representation $\varrho_{t',s'}$ of $D_{2^{k+1}(2r+1)}$, using \eqref{eq:D2kq.2d=} we can get an equivalent representation $\varrho_{t,s}$
with $t=1,\dots,2r$ and $s=0,\dots,2^{k-1}-1$.

Let $x\in H_1(M;\Z)\cong\Ab(D_{2^{k+1}(2r+1)})\cong\mathsf{C}_{2^{k+1}}$ be the generator. By Proposition~\ref{prop:fccs.Dkq} we have
\begin{equation*}
 \widehat{c}_{\varrho_{t_1,s_1},1}(x)=\widehat{c}_{\varrho_{t_2,s_2},1}(x)=\frac{u}{2^{k}},
\end{equation*}
for some integer $u$ such that $0\leq u\leq 2^{k}-1$. We have two cases.

\paragraph{\textbf{Case 1:} $u<2^{k-1}$} By Proposition~\ref{prop:fccs.Dkq} in this case we have
\begin{equation*}
\widehat{c}_{\varrho_{t_1,s_1},1}(x)=\frac{s_1}{2^{k}}=\frac{s_2}{2^{k}}=\widehat{c}_{\varrho_{t_2,s_2},1}(x),
\end{equation*}
thus, $s_1=s_2$ and both $t_1$ and $t_2$ are odd numbers.

\paragraph{\textbf{Case 2:} $u\geq 2^{k-1}$} By Proposition~\ref{prop:fccs.Dkq} in this case we have
\begin{equation*}
 \widehat{c}_{\varrho_{t_1,s_1},1}(x)=\frac{2^{k-1}+s_1}{2^{k}}=\frac{2^{k-1}+s_2}{2^{k}}=\widehat{c}_{\varrho_{t_2,s_2},1}(x),
\end{equation*}
thus, $s_1=s_2$ and both $t_1$ and $t_2$ are even numbers.
\end{proof}

\begin{proposition}\label{prop:xi.D2k2r+1}
	Let $\varrho_{t,s}$ be the two-dimensional irreducible representation of $\pi_1(M)\cong D_{2^{k+1}(2r+1)}$ given in \eqref{eq:D2kq.2dir}. Then
	\begin{multline}\label{eq:xi.gen}
		\tilde{\xi}_{\varrho_{t,s}}(D) =\frac{1}{2^{k+1}(2r+1)}\left[\sum_{l=1}^{2^{k}-1}\frac{(-1)^{tl}2\zeta_{2^{k}}^{ls}-2}{\zeta_{2^{k}}^{l}+\zeta_{2^{k}}^{-l}-(-1)^l2}\right.\\
+\sum_{q=1}^{2r}\sum_{l=0}^{2^{k}-1}\frac{(-1)^{tl}\zeta_{2^{k}}^{ls}(\zeta_{2r+1}^{tq}+\zeta_{2r+1}^{-tq})-2}{\zeta_{2^{k}}^{l}+\zeta_{2^{k}}^{-l}-(-1)^{l}(\zeta_{2r+1}^q+\zeta_{2r+1}^{-q})}\\
\left. -\sum_{q=0}^{2r}\sum_{l=0}^{2^{k}-1}\frac{2}{\zeta_{2^{k+1}}^{(2l+1)}+\zeta_{2^{k+1}}^{-(2l+1)}}\right].
	\end{multline}
\end{proposition}

\begin{proof}
Using presentation \eqref{it:ffi.Dmk}, the elements of $D_{2^{k+1}(2r+1)}$ can be written in the form $x^iy^j$ with $i=0,\dots,2^{k+1}-1$ and $j=0,\dots,2r$.
We can divide the non-identity elements in three sets:
\begin{equation}\label{eq:elem.sets}
\begin{aligned}
x^{2l},&\quad 1\leq l\leq 2^{k}-1,\\
x^{2l}y^q,&\quad 0\leq l\leq 2^{k}-1,\  1\leq q\leq 2r,\\
x^{2l+1}y^q,&\quad 0\leq l\leq 2^{k}-1,\  0\leq q\leq 2r.
\end{aligned}
\end{equation}
The values of the representation $\varrho_{t,s}$ on elements of the sets of elements given in \eqref{eq:elem.sets} are
\begin{equation}\label{eq:rep.sets}
\begin{split}
\varrho_{t,s}(x^{2l})&=\zeta_{2^{k+1}}^{2ls}\begin{pmatrix}
                                        (-1)^{tl} & 0\\
                                        0 & (-1)^{tl}
                                       \end{pmatrix},\\
\varrho_{t,s}(x^{2l}y^q)&=(-1)^{tl}\zeta_{2^{k+1}}^{2ls}\begin{pmatrix}
                                                    \zeta_{2r+1}^{tq} &0\\
                                                    0 & \zeta_{2r+1}^{-tq}
                                                   \end{pmatrix}\\
\varrho_{t,s}(x^{2l+1}y^q)&=\zeta_{2^{k+1}}^{(2l+1)s}\begin{pmatrix}
                                                 0 & (-1)^{tl}\zeta_{2r+1}^{-tq}\\
                                                 (-1)^{t(l+1)}\zeta_{2r+1}^{tq} & 0
                                                \end{pmatrix}.
\end{split}
\end{equation}

The natural representation is $\varsigma=\rho_{1,1}$, thus we have
\begin{equation}\label{eq:defect}
\begin{split}
\mathrm{def}(x^{2l})&=\frac{\zeta_{2^{k+1}}^{2l}}{1-(-1)^l2\zeta_{2^{k+1}}^{2l}+\zeta_{2^{k+1}}^{4l}}=\frac{1}{\zeta_{2^{k+1}}^{2l}+\zeta_{2^{k+1}}^{-2l}-(-1)^l2}\\
\mathrm{def}(x^{2l}y^q)&=\frac{\zeta_{2^{k+1}}^{2l}}{1-(-1)^l\zeta_{2^{k+1}}^{2l}(\zeta_{2r+1}^q+\zeta_{2r+1}^{-q})+\zeta_{2^{k+1}}^{4l}}\\
&=\frac{1}{\zeta_{2^{k+1}}^{2l}+\zeta_{2^{k+1}}^{-2l}-(-1)^l(\zeta_{2r+1}^q+\zeta_{2r+1}^{-q})}\\
\mathrm{def}(x^{2l+1}y^q)&=\frac{\zeta_{2^{k+1}}^{(2l+1)}}{1+\zeta_{2^{k+1}}^{2(2l+1)}}=\frac{1}{\zeta_{2^{k+1}}^{(2l+1)}+\zeta_{2^{k+1}}^{-(2l+1)}}.
\end{split}
\end{equation}

By Theorem~\ref{thm:xi.ssf}, \eqref{eq:rep.sets} and \eqref{eq:defect} we have
\begin{multline*}\label{eq:xi.gen}
\tilde{\xi}_{\varrho_{t,s}}(D)=\frac{1}{|D_{2^{k+1}(2r+1)}|}\left[\sum_{l=1}^{2^{k}-1}\frac{(-1)^{tl}2\zeta_{2^{k+1}}^{2ls}-2}{\zeta_{2^{k+1}}^{2l}+\zeta_{2^{k+1}}^{-2l}-(-1)^l2}\right.\\
+\sum_{q=1}^{2r}\sum_{l=0}^{2^{k}-1}\frac{(-1)^{tl}\zeta_{2^{k+1}}^{2ls}(\zeta_{2r+1}^{tq}+\zeta_{2r+1}^{-tq})-2}{\zeta_{2^{k+1}}^{2l}+\zeta_{2^{k+1}}^{-2l}-(-1)^{l}(\zeta_{2r+1}^q+\zeta_{2r+1}^{-q})}\\
\left. +\sum_{q=0}^{2r}\sum_{l=0}^{2^{k}-1}\frac{-2}{\zeta_{2^{k+1}}^{(2l+1)}+\zeta_{2^{k+1}}^{-(2l+1)}}\right].
\end{multline*}
Since $\zeta_{2^{k+1}}^{2}=\zeta_{2^{k}}$ we obtain the expression of the proposition.
\end{proof}

In order to prove that the two-dimensional irreducible representations of $\pi_1(M)\cong D_{2^{k+1}(2r+1)}$ are classified by their vector of CCS-numbers,
we need to prove that if the two-dimensional irreducible representations $\varrho_{t_1,s_1}$ and $\varrho_{t_2,s_2}$ of $\pi_1(M)\cong D_{2^{k+1}(2r+1)}$ have the same vector of CCS-numbers then
$\varrho_{t_1,s_1}=\varrho_{t_2,s_2}$. The other implication is trivial.

\begin{proposition}\label{prop:D2k2r.2d.sv}
Let $\varrho_{t_1,s_1}$ and $\varrho_{t_2,s_2}$ be two $2$-dimensional irreducible representations of $D_{2^{k+1}(2r+1)}$ with $1\leq t_1,t_2\leq 2r$ and $0\leq s_1,s_2\leq 2^{k-1}-1$.
Suppose $\varrho_{t_1,s_1}$ and $\varrho_{t_2,s_2}$ have the same vector of CCS-numbers, then $s_1=s_2$, $t_1\equiv t_2\mod 2$ and $\tilde{\xi}_{\varrho_{t_1,s_1}}(D)=\tilde{\xi}_{\varrho_{t_2,s_2}}(D)$.
\end{proposition}

\begin{proof}
Suppose that $\varrho_{t_1,s_1}$ and $\varrho_{t_2,s_2}$ have the same vector of CCS-numbers, that is, they have the same first and second CCS-numbers.
Thus, they have the same first Cheeger-Chern-Simons class and by Theorem~\ref{thm:rhs}-\eqref{it:c1.r.det}
\begin{equation*}
\widehat{c}_{\det(\varrho_{t_1,s_1}),1}=\widehat{c}_{\varrho_{t_1,s_1},1}=\widehat{c}_{\varrho_{t_2,s_2},1}=\widehat{c}_{\det(\varrho_{t_2,s_2}),1}.
\end{equation*}
Hence, by Corollary~\ref{cor:fccsn=s} and Remark~\ref{rem:1dm.D2kq}
\begin{align}
s_1&=s_2,\\
t_1&\equiv t_2\mod 2,\\
\det(\varrho_{t_1,s_1})&=\det(\varrho_{t_2,s_2}).\label{eq:det=det}
\end{align}
From Theorem~\ref{thm:rhs}-\eqref{it:top.triv} and Proposition~\ref{prop:fccs.Dkq} the topologically trivial two-dimensional irreducible representations are $\varrho_{t,0}$ with $t$ odd.
Thus, if $s_1=s_2=0$ and both $t_1$ and $t_2$ are odd, by \eqref{eq:CCS2=xi} we have
\begin{equation*}
\tilde{\xi}_{\varrho_{t_1,0}}(D)=\widehat{c}_{\varrho_{t_1,0},2}([M])= \widehat{c}_{\varrho_{t_2,0},2}([M])=\tilde{\xi}_{\varrho_{t_2,0}}(D).
\end{equation*}
For the other cases, by Theorem~\ref{thm:rhs}-\eqref{it:xi.non.triv} we have that
\begin{equation}
\tilde{\xi}_{\varrho_{t_1,s_1}}(D)-\tilde\xi_{\det(\varrho_{t_1,s_1})}(D)=\widehat{c}_{\varrho_{t_1,s_1},2}([M])= \widehat{c}_{\varrho_{t_2,s_2},2}([M])=\tilde{\xi}_{\varrho_{t_2,s_2}}(D)-\tilde\xi_{\det(\varrho_{t_2,s_2})}(D).
\end{equation}
By \eqref{eq:det=det} we have that $\tilde\xi_{\det(\varrho_{t_1,s_1})}(D)=\tilde\xi_{\det(\varrho_{t_2,s_2})}(D)$, hence $\tilde{\xi}_{\varrho_{t_1,s_1}}(D)=\tilde{\xi}_{\varrho_{t_2,s_2}}(D)$.
\end{proof}

Hence, by Proposition~\ref{prop:D2k2r.2d.sv}, to prove that the two-dimensional irreducible representations of $\pi_1(M)\cong D_{2^{k+1}(2r+1)}$ are classified by their vector of CCS-numbers we need to prove the following conjecture.

\begin{conjecture}\label{conj:Dk2r.2d}
Let $\varrho_{t_1,s}$ and $\varrho_{t_2,s}$ be two $2$-dimensional irreducible representations of $D_{2^{k+1}(2r+1)}$ such that $1\leq t_1,t_2\leq 2r$, $t_1\equiv t_2\mod 2$ and $0\leq s\leq 2^{k-1}-1$.
If $\tilde{\xi}_{\varrho_{t_1,s}}(D)=\tilde{\xi}_{\varrho_{t_2,s}}(D)$ then $t_1=t_2$.
\end{conjecture}

So far, we have not been able to prove Conjecture~\ref{conj:Dk2r.2d} using \eqref{eq:xi.gen}, so this case is still missing.

Conjecture~\ref{conj:Dk2r.2d} is satisfied in the examples we have computed, for instance, in Examples~\ref{ex:8.5} and \ref{ex:16.5}
we can see in Tables~\ref{tb:c1.8.5} and \ref{tb:c1.16.5} that, as Corollary~\ref{cor:fccsn=s} says, the representations with the same first CSS-number have the same $s$ and $t$ with the same parity.
Thus, the first CCS-numbers are not sufficient to distinguish these irreducible representations. But the $\tilde\xi$-invariant given in Tables~\ref{tb:xi.8.5} and \ref{tb:xi.16.5} distinguish them.

\begin{example}\label{ex:8.5}
Consider the group $D_{8(5)}$. It has $8$ different two-dimensional irreducible representations $\varrho_{t,s}$.
Their first CCS-numbers are:
\begin{table}[H]
\begin{equation*}
\setlength{\extrarowheight}{3pt}
\begin{array}{|c|c|c|}\hline
4\cdot\widehat{c}_{\varrho_{t,s},1}(x)&s=0&s=1\\
\hline
t=1 & 0 & 1\\\hline
t=2 & 2 & 3\\\hline
t=3 & 0 & 1\\\hline
t=4 & 2 & 3\\\hline
\end{array}
\end{equation*}
\caption{Values of $\widehat{c}_{\varrho_{t,s},1}(x)$ for the group $D_{8(5)}$.}\label{tb:c1.8.5}
\end{table}
Their $\tilde\xi$-invariants are:
\begin{table}[H]
\begin{equation*}
\begin{array}{|c|c|c|}\hline
40\cdot\tilde{\xi}_{\rho_{t,s}}(D) & s=0 & s=1 \\\hline\hline
t=1 & -4 & -9 \\\hline
t=2 & -16 & -1 \\\hline
t=3 & 4 & -1 \\\hline
t=4 & 16 & -9 \\\hline
\end{array}
\end{equation*}
\caption{Values of $\tilde{\xi}_{\rho_{t,s}}(D)$ for the group $D_{8(5)}$.}\label{tb:xi.8.5}
\end{table}
\end{example}

\begin{example}\label{ex:16.5}
Consider the group $D_{16(5)}$. It has $16$ different two-dimensional irreducible representations $\varrho_{t,s}$.
Their first CCS-numbers are:
\begin{table}[H]
\begin{equation*}
\setlength{\extrarowheight}{3pt}
\begin{array}{|c|c|c|c|c|}\hline
8\cdot\widehat{c}_{\varrho_{t,s},1}&s=0&s=1&s=2&s=3\\
\hline
t=1 &0&1&2&3\\\hline
t=2 &4&5&6&7\\\hline
t=3 &0&1&2&3\\\hline
t=4 &4&5&6&7\\\hline
\end{array}
\end{equation*}
\caption{Values of $\widehat{c}_{\varrho_{t,s},1}(x)$ for the group $D_{16(5)}$.}\label{tb:c1.16.5}
\end{table}
Their $\tilde\xi$-invariants are:
\begin{table}[H]
\begin{equation*}
\setlength{\extrarowheight}{3pt}
\begin{array}{|c|c|c|c|c|}\hline
80\cdot\tilde{\xi}_{\rho_{t,s}}(D)&s=0&s=1&s=2&s=3\\
\hline
t=1 &96&31&76&71\\\hline
t=2 &64&39&44&79 \\\hline
t=3 &64&79&44&39\\\hline
t=4 &96&71&76&31\\\hline
\end{array}
\end{equation*}
\caption{Values of $\tilde{\xi}_{\rho_{t,s}}(D)$ for the group $D_{16(5)}$.}\label{tb:xi.16.5}
\end{table}
\end{example}

There are groups $D_{2^{k+1}(2r+1)}$ for which some of the two-dimensional representations $\rho_{t,s}$ have the same $\tilde\xi$-invariant.
However, the representations $\rho_{t,s}$ are classified by their first CCS-numbers.
For instance, in Example~\ref{ex:16.3} we can see in Table~\ref{tb:xi.16.3} that $\tilde{\xi}_{\varrho_{1,s}}(D)=\tilde{\xi}_{\varrho_{2,s}}(D)$ with $s=0,2$,
but in Table~\ref{tb:c1.16.3} we have that $\widehat{c}_{\varrho_{1,s},1}(x)\neq\widehat{c}_{\varrho_{2,s},1}(x)$ with $s=0,2$.

\begin{example}\label{ex:16.3}
Consider the group $D_{16(3)}$. It has $8$ different two-dimensional irreducible representations $\varrho_{t,s}$. Their first CCS-numbers are:
\begin{table}[H]
\begin{equation*}
\setlength{\extrarowheight}{3pt}
\begin{array}{|c|c|c|c|c|}\hline
8\cdot\widehat{c}_{\varrho_{t,s},1}(x)&s=0 &s=1&s=2&s=3\\
\hline
t=1 & 0 & 1 & 2 & 3\\\hline
t=2 & 4 & 5 & 6 & 7\\[3pt]\hline
\end{array}
\end{equation*}
\caption{Values of $\widehat{c}_{\varrho_{t,s},1}(x)$ for the group $D_{16(3)}$.}\label{tb:c1.16.3}
\end{table}
The $\tilde\xi$-invariants of   $\varrho_{t,s}$ are given by
\begin{table}[H]
\begin{equation*}
\setlength{\extrarowheight}{3pt}
\begin{array}{|c|c|c|c|c|}\hline
48\cdot\tilde{\xi}_{\rho_{t,s}}(D)&s=0&s=1&s=2&s=3\\
\hline
t=1&-32 & -17&-20&-41\\\hline
t=2&-32 & -41 &-20&-17 \\[3pt]\hline
\end{array}
\end{equation*}
\caption{Values of $\tilde{\xi}_{\rho_{t,s}}(D)$ for the group $D_{16(3)}$. }\label{tb:xi.16.3}
\end{table}
\end{example}

\subsection{Classification of irreducible representations of $P'_{8\cdot 3^k}$} 

The group $P'_{8\cdot 3^k}$ can be seen as a subgroup of $\U$ and by Theorem~\ref{thm:D.T} it corresponds to the group $\mathbb{T}_{3^{k-1}}$ in Theorem~\ref{thm:fssu2}.
Consider the spherical $3$-manifold $M=\mathbb{S}^3/\mathbb{T}_{3^{k-1}}$ given by the quotient of the action of $\mathbb{T}_{3^{k-1}}$ on $\mathbb{S}^3$. Hence we have $\pi_1(M)\cong\mathbb{T}_{3^{k-1}}\cong P'_{8\cdot 3^k}$.

\begin{remark}\label{rem:1dr.Pp}
From Subsection~\ref{ssec:Pp83k} the group $P'_{8\cdot 3^k}$ has $3^k$ one-dimensional irreducible representations, which correspond to the irreducible representations of its abelianization $\mathsf{C}_{3^k}$; thus, the first CCS-numbers of one-dimensional representations are given as in Table~\ref{tbl:nq} with $n=3^k$.
So they are classified by the first CCS-number and their vector of CCS-numbers is of the form $(1;\widehat{c}_{\alpha_j,1}(x);0)$.
\end{remark}

\begin{proposition}\label{prop:Pp83k.class}
	The irreducible representations of the group $P'_{8\cdot 3^k}$ are classified by their rank and first CCS-number.
\end{proposition}

\begin{proof}
We have three cases given by the rank.
\begin{enumerate}
\item By Remark~\ref{rem:1dr.Pp} one-dimensional representations are classified by their first CCS-number.\label{it:case1d}

\item Let $\varrho_{s_1}$ and $\varrho_{s_2}$, with $s_1,s_2=0,\dots,3^{k}-1$, be two-dimensional irreducible representations, so they are of the form given in \eqref{eq:2drPp83k}.
Suppose they have the same first CCS-number, thus by case~\eqref{it:case1d} their determinats coincide. From \eqref{eq:2drPp83k} the determinants are
\begin{equation}
\begin{aligned}
\det\varrho_{s_i}(x)&=1, & \det\varrho_{s_i}(y)&=1, & \det\varrho_{s_i}(z)&=\zeta_{3^k}^{2s_i}, \quad i=1,2.\\
\end{aligned}
\end{equation}
The equality $\det\varrho_{s_1}(z)=\det\varrho_{s_2}(z)$ implies that $\zeta_{3^k}^{2s_1}=\zeta_{3^k}^{2s_2}$; since $s_1,s_2\in\{0,\dots,3^{k}-1\}$ then $s_1=s_2$.

\item Let $\varsigma_{s_1}$ and $\varsigma_{s_2}$, with $s_1,s_2=0,\dots,3^{k-1}-1$, be three-dimensional irreducible representations given in \eqref{eq:3drPp83k}.
As before, if they have the same first CCS-number, by case~\eqref{it:case1d} their determinants coincide and from \eqref{eq:3drPp83k}
\begin{equation}
	\begin{aligned}
		\det\varsigma_{s_i}(x)&=-1, & \det\varsigma_{s_i}(y)&=1, & \det\varsigma_{s_i}(z)&=\zeta_{3^k}^{s_i}, \quad i=1,2.
	\end{aligned}
\end{equation}
From $\det\varsigma_{s_1}(z)=\det\varsigma_{s_2}(z)$ we have that $\zeta_{3^k}^{s_1}=\zeta_{3^k}^{s_2}$, so $s_1=s_2$.\qedhere

\end{enumerate}
\end{proof}

\subsection{Classification of irreducible representations of $\Gamma\times\mathsf{C}_m$}\label{ssec:GxC}

Our aim is to classify flat vector bundles over spherical $3$-manifolds, or equivalently, the corresponding irreducible representations of their fundamental group
given by \eqref{eq:Flat.Rep}, by their vector of CCS-numbers (equivalently, their CCS-triple).
Let $M$ be a spherical $3$-manifolds, its fundamental group $\Gamma$ is listed in Theorem~\ref{thm:fsgfi}. 
If $\Gamma$ is one of the groups $\mathsf{BT}$, $\mathsf{BO}$ or $\mathsf{BI}$ (groups \eqref{it:ffi.P24}, \eqref{it:ffi.P48} and \eqref{it:ffi.P120} in Theorem~\ref{thm:fsgfi}),
they are finite subgroups of $\SU$, and the fact that their irreducible representations are classified by their vector of CCS-numbers follows from the results in \cite[\S5.3, \S5.4]{arciniegaetal:CCSC} which we summarize in Table~\ref{tb:fcn}, where $[M]$ is the fundamental class of $M$ and $\bar{c}$ the generator of $H_1(M;\Z)\cong\Ab(\Gamma)$ (see Table~\ref{tb:AbG})
\begin{table}[H]
	\begin{equation*}
		\setlength{\extrarowheight}{3pt}
		\begin{array}{|c|c|c|c|c|c|c|c|c|c|}\hline
			\mathsf{BT} & \alpha_1 & \alpha_2 & \alpha_3 & \alpha_4 & \alpha_5 & \alpha_6 & \alpha_7 & & \\\hline
			\widehat{c}_{\alpha,1}(\bar{c}) & 0 & \frac{2}{3} & \frac{1}{3} & 0 & \frac{1}{3} & \frac{2}{3} & 0 & & \\[3pt]\hline
			\widehat{c}_{\alpha,2}([M]) & 0 & 0 & 0 & \frac{1}{24} & \frac{3}{8} & \frac{3}{8} & \frac{1}{6} & & \\[4pt]\hline\hline			
			\mathsf{BO} & \alpha_1 & \alpha_2 & \alpha_3 & \alpha_4 & \alpha_5 & \alpha_6 & \alpha_7 & \alpha_8 & \\\hline
			\widehat{c}_{\alpha,1}(\bar{c}) & 0 & \frac{1}{2} & \frac{1}{2} & 0 & 0 & \frac{1}{2} & 0 & 0 & \\[3pt]\hline
			\widehat{c}_{\alpha,2}([M]) & 0 & 0 & \frac{1}{3} & \frac{1}{48} & \frac{25}{48} & \frac{7}{12} & \frac{1}{12} & \frac{5}{24} & \\[4pt]\hline\hline			
			\mathsf{BI} & \alpha_1 & \alpha_2 & \alpha_3 & \alpha_4 & \alpha_5 & \alpha_6 & \alpha_7 & \alpha_8 & \alpha_9 \\\hline
			\widehat{c}_{\alpha,1}(1) & 0 & 0 & 0 & 0 & 0 & 0 & 0 & 0 & 0 \\\hline
			\widehat{c}_{\alpha,2}([M]) & 0 & \frac{1}{120} & \frac{49}{120} & \frac{19}{30} & \frac{1}{30} & \frac{5}{6} & \frac{1}{12} & \frac{1}{6} & \frac{7}{24} \\[4pt]\hline			
		\end{array}
	\end{equation*}
	\caption{CCS-numbers of irreducible representations of $\mathsf{BT}$, $\mathsf{BO}$ and $\mathsf{BI}$}\label{tb:fcn}
\end{table}

If $\Gamma$ is $\mathsf{C}_n$, $\mathsf{BD}_{2q}$ or $P'_{8\cdot 3^k}$ (families \eqref{it:ffi.Q} or \eqref{it:ffi.Pk} in Theorem~\ref{thm:fsgfi}) the classification of irreducible representations by their vector of CCS-numbers follows from Corollary~\ref{cor:vccs.lens}, Remark~\ref{rem:1dr.vcssn}, Proposition~\ref{prop:2dir.vcssn.cls} and Proposition~\ref{prop:Pp83k.class}.

If $\Gamma$ is $D_{2^k(2r+1)}$ (family \eqref{it:ffi.Dmk} in Theorem~\ref{thm:fsgfi}) by Remark~\ref{rem:1dm.D2kq} one-dimensional irreducible representations are classified by their vector of CCS-numbers. The case of two-dimensional irreducible representations is missing, but we conjecture (Conjecture~\ref{conj:Dk2r.2d}) that they are also classified by their vector of CCS-numbers.

Here we deal with spherical $3$-manifolds $M$ whose fundamental group $\pi_1(M)=G$ is in family \eqref{it:Gammaxcyc} in Theorem~\ref{thm:fsgfi} of groups of the form $G=\Gamma\times\mathsf{C}_m$, where $\Gamma$ is $\mathsf{BD}_{2q}$, $\mathsf{BT}$, $\mathsf{BO}$, $\mathsf{BI}$, $D_{2^k(2r+1)}$ or $P'_{8\cdot 3^k}$ (groups \eqref{it:ffi.Q}-\eqref{it:ffi.Pk} in Theorem~\ref{thm:fsgfi}).
But, as we mentioned above, only for one-dimensional representations in the case of $\Gamma=D_{2^k(2r+1)}$.

We shall use Chern and Cheeger-Chern-Simons classes in $\BG[G]$ defined in Remark~\ref{rem:G.M.ccs} and the following lemma.
\begin{lemma}[{\cite[\S5.5]{arciniegaetal:CCSC}}]\label{lem:c=ccs}
Let $G$ be a finite group and $\rho\colon G\to\GLn{n}$ a representation. Then there is an isomorphism
$H^{j}(\BG[G];\Z)\cong H^{j-1}(\BG[G]; \C/\Z)$ for $j\geq 2$, under which, the $k$-th Chern class $c_k(\rho)$ and the $k$-th Cheeger-Chern-Simons class $\widehat{c}_{k}(\rho)$ of $\rho$ in $\BG[G]$ correspond to each other
\begin{equation}\label{eq:crk.ckr}
	\begin{split}
		H^{j}(\BG[G];\Z)&\cong H^{j-1}(\BG[G]; \C/\Z)\\
		c_k(\rho)&\leftrightarrow \widehat{c}_k(\rho).
	\end{split}
\end{equation}
\end{lemma}

\begin{proof}
By the Universal Coefficients Theorem, the fact that $\C$ is a divisible abelian group and that
the homology groups $H_j(\BG[G];\Z)$ are torsion, one can prove that $H^j(\BG[G];\C)=0$ for all positive integer $j$.
Using this, in the cohomology long exact sequence 
\begin{equation}\label{eq:coef.CheegerSimons2}
	\dots\to H^{2k-1}(\BG[G];\C/\Z) \xrightarrow{q} H^{2k}(\BG[G];\Z) \xrightarrow{r} H^{2k}(\BG[G];\C) \xrightarrow{p_\Z}\cdots
\end{equation}
induced by the short exact sequence of coefficients $0 \to \Z \to \C \to \C/\Z \to 0$, we get the result.
\end{proof}

\begin{remark}\label{rem:Chern.triple}
To prove that the irreducible representations of $\pi_1(M)=G$ are classified by their vector of CCS-numbers, by Remark~\ref{rem:CCSt.vCCSn} it is equivalent to prove that they are classified by their CSS-triple $[\dim\rho; \widehat{c}_{\rho,1}; \widehat{c}_{\rho,2}]$, which in turn, by Remark~\ref{rem:G.M.ccs}
it is equivalent to prove that they are classified by their CCS-triple $[\dim\rho; \widehat{c}_{1}(\rho); \widehat{c}_{2}(\rho)]$ in $\BG[G]$.
By Lemma~\ref{lem:c=ccs}, it is equivalent to prove that they are classified by their triple
of Chern classes $[\dim\rho; c_{1}(\rho); c_{2}(\rho)]$ in $\BG[G]$.
\end{remark}

We need a technical lemma.
\begin{lemma}\label{lem:no.divisors}
Let $\Gamma$ be a group in the families \eqref{it:ffi.Q}-\eqref{it:ffi.Pk} in Theorem~\ref{thm:fsgfi}. 
Let $n>1$ be the rank of an irreducible representation of $\Gamma$ and $m\in\mathbb{N}$ such that $\gcd(|\Gamma|,m)=1$.
Then $n$ does not divide $m$. Thus, if $nc=0\in\Z_m$, then $c=0$.
\end{lemma}

\begin{proof}
We need to analyse each of the families \eqref{it:ffi.Q}-\eqref{it:ffi.Pk} in Theorem~\ref{thm:fsgfi}.
\begin{description}
\item[$\mathsf{BD}_{2q}$] In this case $|\mathsf{BD}_{2q}|=4q$, $n=2$, since $\gcd(4q,m)=1$, then $2$ does not divide $m$.

\item[$\mathsf{BT}$] In this case $|\mathsf{BT}|=24$, $n=2,3$, since $\gcd(24,m)=1$, then $2$ and $3$ do not divide $m$.

\item[$\mathsf{BO}$] In this case $|\mathsf{BO}|=48$, $n=2,3,4$, since $\gcd(48,m)=1$, then $2$, $3$ and $4$ do not divide $m$.

\item[$\mathsf{BI}$]  In this case $|\mathsf{BI}|=120$, $n=2,3,4,5,6$, since $\gcd(120,m)=1$, then $2$, $3$, $4$, $5$ and $6$ do not divide $m$.

\item[$D_{2^k(2r+1)}$] In this case $|D_{2^k(2r+1)}|=2^k(2r+1)$, $n=2$, since $\gcd(2^k(2r+1),m)=1$, then $2$ does not divide $m$.

\item[$P'_{8\cdot 3^k}$] In this case $|P'_{8\cdot 3^k}|=8\cdot 3^k$, $n=2,3$, since $\gcd(8\cdot 3^k,m)=1$, then $2$ and $3$ do not divide $m$.\qedhere
\end{description}
\end{proof}

\begin{proposition}\label{prop:G.product.class}
Consider a group of the form $\Gamma\times\mathsf{C}_m$, where $\Gamma$ is a group in the families
\eqref{it:ffi.Q}-\eqref{it:ffi.Pk} in Theorem~\ref{thm:fsgfi} and $\gcd(|\Gamma|,m)=1$.
The irreducible representations of $\Gamma\times\mathsf{C}_m$ are classified by their
triple of Chern classes $[\dim\rho; c_{1}(\rho); c_{2}(\rho)]$ in $\BG[(\Gamma\times\mathsf{C}_m)]$.
In the case when $\Gamma=D_{2^k(2r+1)}$ we only consider one-dimensional representations.
\end{proposition}

\begin{proof}
Let $\rho_1$ and $\rho_2$ be two irreducible representations of $\Gamma\times\mathsf{C}_m$. We have to prove that $\rank\rho_1=\rank\rho_2$, $c_{1}(\rho_1)=c_{1}(\rho_2)$
and $c_{2}(\rho_1)=c_{2}(\rho_2)$ implies that $\rho_1\cong\rho_2$. The other implication is trivial. The irreducible representations $\rho_i$, with $i=1,2$, can be written as
\begin{equation}\label{eq:tensor}
\rho_i=\rho_i^{\Gamma}\otimes\rho_i^{\mathsf{C}_m},
\end{equation}
where $\rho_i^{\Gamma}$ and $\rho_i^{\mathsf{C}_m}$ are, respectively, irreducible representation of $\Gamma$ and $\mathsf{C}_m$. Using this decomposition we get (see Appendix~\ref{app:tensor})
\begin{gather}
c_1(\rho_i)=c_1(\rho_i^{\Gamma}\otimes\rho_i^{\mathsf{C}_m})=c_1(\rho_i^{\Gamma})+\rank(\rho_i^{\Gamma})c_1(\rho_i^{\mathsf{C}_m})\label{eq:c1.tensor}\\
\scriptstyle c_2(\rho_i)=c_2(\rho_i^{\Gamma}\otimes\rho_i^{\mathsf{C}_m})=c_2(\rho_i^{\Gamma})+(\dim(\rho_i^{\Gamma})-1)c_1(\rho_i^{\Gamma})\cdot c_1(\rho_i^{\mathsf{C}_m})+\binom{\rank(\rho_i^{\Gamma})}{2}c_1(\rho_i^{\mathsf{C}_m})^2.\label{eq:c2.tensor}
\end{gather}
By hypothesis $\rank\rho_1=\rank\rho_2=n$ and since
\begin{align*}
	\rank\rho_1= \rank(\rho^{\Gamma}_1 \otimes \rho^{\mathsf{C}_m}_1) = \rank(\rho^{\Gamma}_1) \cdot \rank(\rho^{\mathsf{C}_m}_1 ) = \rank(\rho^{\Gamma}_1 ) \\
	\rank\rho_2= \rank(\rho^{\Gamma}_2 \otimes \rho^{\mathsf{C}_m}_2) = \rank(\rho^{\Gamma}_2) \cdot \rank(\rho^{\mathsf{C}_m}_2 ) = \rank(\rho^{\Gamma}_2 )
\end{align*}
we have that 
\begin{equation}\label{eq:ranks}
\rank(\rho_1^{\Gamma})=\rank(\rho_2^{\Gamma})=n.
\end{equation}
Also by hypothesis $c_1(\rho_1)=c_1(\rho_2)$, so by \eqref{eq:c1.tensor}
\begin{equation*}
c_1(\rho_1^{\Gamma})+nc_1(\rho_1^{\mathsf{C}_m})=
c_1(\rho_2^{\Gamma})+nc_1(\rho_2^{\mathsf{C}_m}),
\end{equation*}
which is equivalent to
\begin{equation}\label{eq:separated}
0=\bigl(c_1(\rho_1^{\Gamma})-c_1(\rho_2^{\Gamma})\bigr)+n\bigl(\rho_1^{\mathsf{C}_m})-c_1(\rho_2^{\mathsf{C}_m})\bigr)\in H^2(\Gamma\times\mathsf{C}_m;\Z).
\end{equation}
Every group $G$ which acts freely on $\mathbb{S}^3$ has periodic Tate cohomology $\hat{H}^*(G;\Z)$ of period $4$ \cite[Corollary~0.3-(a)]{Davis-Milgram:SSSFP},
that is $\hat{H}^{k}(G;\Z)\cong \hat{H}^{k+4}(G;\Z)$, in particular
\begin{equation}\label{eq:per.coho}
H_1(G;\Z)=\hat{H}^{-2}(G;\Z)\cong \hat{H}^{2}(G;\Z)=H^{2}(G;\Z).
\end{equation}
Hence $\Gamma$, $\mathsf{C}_m$ and $\Gamma\times\mathsf{C}_m$ have periodic Tate cohomology of period $4$ and they satisfy \eqref{eq:per.coho}.
By Küneth theorem we have
\begin{equation*}\scriptstyle
0\to \bigoplus\limits_{i+j=1}H_i(\Gamma;\Z)\otimes H_j(\mathsf{C}_m;\Z)\to H_1(\Gamma\times\mathsf{C}_m;\Z)\to\bigoplus\limits_{i+j=1}\mathrm{Tor}^{\Z}_1\bigl(H_i(\Gamma;\Z),H_j(\mathsf{C}_m;\Z)\bigr)
\to0.
\end{equation*}
Therefore by \eqref{eq:per.coho}
\begin{equation}\label{eq:dir.prod}
H^2(\Gamma\times\mathsf{C}_m;\Z)\cong H_1(\Gamma\times\mathsf{C}_m;\Z)\cong H_1(\Gamma;\Z)\oplus H_1(\mathsf{C}_m;\Z).
\end{equation}
Moreover, by \eqref{eq:separated} and \eqref{eq:per.coho} we have that
\begin{equation}\label{eq:separated2}
\begin{aligned}
0=\bigl(c_1(\rho_1^{\Gamma})-c_1(\rho_2^{\Gamma})\bigr)&\in H^2(\Gamma;\Z)\cong H_1(\Gamma;\Z)\\
0=n\bigl(c_1(\rho_1^{\mathsf{C}_m})-c_1(\rho_2^{\mathsf{C}_m})\bigr)&\in H^2(\mathsf{C}_m;\Z)\cong H_1(\mathsf{C}_m;\Z).
\end{aligned}
\end{equation}
By \eqref{eq:separated} and \eqref{eq:separated2} and Lemma~\ref{lem:no.divisors} we get
\begin{align}
c_1(\rho_1^{\Gamma})&=c_1(\rho_2^{\Gamma})\label{eq:c1Gamma=}\\
c_1(\rho_1^{\mathsf{C}_m})&=c_1(\rho_2^{\mathsf{C}_m}).\label{eq:c1Cl=}
\end{align}
The representations $\rho_i^{\mathsf{C}_m}$ have dimension one. By Remark~\ref{rem:irr.rep.Cn.class} and Remark~\ref{rem:Chern.triple},
one dimensional representations of $\mathsf{C}_m$ are classified by their first Chern class $c_i(\rho_i^{\mathsf{C}_m})$ in $\BG[\mathsf{C}_m]$.
Hence, by \eqref{eq:c1Cl=} we have that 
\begin{equation}\label{eq:part.Cm}
\rho_1^{\mathsf{C}_m}\cong\rho_2^{\mathsf{C}_m}.
\end{equation}

By hypothesis $c_2(\rho_1)=c_2(\rho_2)$, then \eqref{eq:c2.tensor} and \eqref{eq:ranks} imply
\begin{multline}\label{eq:c2=}
c_2(\rho_1^{\Gamma})+(n-1)c_1(\rho_1^{\Gamma})\cdot c_1(\rho_1^{\mathsf{C}_m})+\binom{n}{2}c_1(\rho_1^{\mathsf{C}_m})^2\\
=c_2(\rho_2^{\Gamma})+(n-1)c_1(\rho_2^{\Gamma})\cdot c_1(\rho_2^{\mathsf{C}_m})+\binom{n}{2}c_1(\rho_2^{\mathsf{C}_m})^2.
\end{multline}
By \eqref{eq:ranks}, \eqref{eq:c1Gamma=}, \eqref{eq:c1Cl=} and \eqref{eq:c2.tensor} we get
\begin{equation}\label{eq:c2Gamma=}
c_2(\rho_1^{\Gamma})=c_2(\rho_2^{\Gamma}).
\end{equation}
The irreducible representations of $\Gamma$ are classified by their vectors of CCS-numbers or equivalently, by Remark~\ref{rem:Chern.triple},
they are classified by their triple of Chern classes $[\dim\rho^{\Gamma}; c_{1}(\rho^{\Gamma}); c_{2}(\rho^{\Gamma})]$ in $\BG[G]$.
Hence, by \eqref{eq:ranks}, \eqref{eq:c1Gamma=} and \eqref{eq:c2Gamma=} we have that
\begin{equation}\label{eq:part.Gamma}
\rho_1^{\Gamma}\cong\rho_2^{\Gamma}.
\end{equation}
By \eqref{eq:tensor}, \eqref{eq:part.Cm} and \eqref{eq:part.Gamma} we have that $\rho_1\cong\rho_2$, this finishes the proof.
\end{proof}

Putting together the results given in Table~\ref{tb:fcn}, Corollary~\ref{cor:vccs.lens}, Remark~\ref{rem:1dr.vcssn} Proposition~\ref{prop:2dir.vcssn.cls}, Remark~\ref{rem:1dm.D2kq}, 
Proposition~\ref{prop:Pp83k.class} and Proposition~\ref{prop:G.product.class} we obtain the following theorem.

\begin{theorem}\label{th:main.class}
Let $\Gamma$ be a group listed in Theorem~\ref{thm:fsgfi} and let $\varsigma\colon\Gamma\to U(2)$ be a faithful fixed-point free unitary representation.
Consider the spherical $3$-manifold $M=\mathbb{S}^3/\varsigma(\Gamma)$. The irreducible representations of $\pi_1(M)\cong\Gamma$ are classified by their vector of CCS-numbers.
In the case of $\Gamma=D_{2^k(2r+1)}\times\mathsf{C}_l$ with $\gcd(2^k(2r+1),l)=1$ we only consider one-dimensional representations.
\end{theorem}

\begin{remark}
By Remark~\ref{rem:CCSt.vCCSn} and correspondence \eqref{eq:Flat.Rep} Theorem~\ref{th:main.class} is equivalent to Theorem~\ref{Thm2}
\end{remark}

\section{Singularity Theory}\label{sec:sing}

In this section we recall basics on reflexive modules, full sheaves and topological invariants of normal surface singularities. We assume basic familiarity with these objects,
see \cite{Bruns-Herzog:CMR,Ishii:Sing,Nemethi:5lect,Nemethi:NSS} for more details. From now on, we denote by $(X,x)$ the germ of a two dimensional quotient singularity, i.e., there exists a small finite subgroup $\Gamma \subset \GLn[2]{\C}$ such that $(X,x) \cong (\C^2 / \Gamma,0)$ as germs. 
We denote by $\Ss{X}$ the sheaf of holomorphic functions on the complex space $X$.

\subsection{Minimal resolution and link}

Let $(X,x)$ be the germ of a two-dimensional quotient singularity. Let
\begin{equation*}
	\pi\colon\tilde{X}\to X
\end{equation*}
be the \textit{minimal resolution of $(X,x)$}, i.e., a proper holomorphic map from a smooth surface $\Rs$ to a given representative of $(X,x)$ such that $\pi$ restricted to the complement of $E=\pi^{-1}(x)$ is a biholomorphism and $E$ does not contain a projective line $\mathbb{P}^1$ with self intersection equal to -1.

\begin{definition}
The \textit{geometric genus} $p_g$ of $X$ is the dimension of the $\C$-vector space $H^1(\Rs,\Ss{\Rs})$. We say that $(X,x)$ is a \emph{rational singularity} if $p_g=0$.
\end{definition}

\begin{remark}\label{rmk:quotient.isrational}
Every quotient surface singularity is rational \cite[Corollary 7.4.10]{Ishii:Sing}.
\end{remark}

Let $\mathrm{Pic}(\Rs)=H^1(\Rs,\Ss{\Rs}^*)$ denote the Picard group of $\Rs$, the group of isomorphism classes of analytic line bundles on $\Rs$. 
Since $(X,x)$ is rational, the homomorphism 
\begin{equation}\label{eq:Pic.Chern}
c_1\colon\mathrm{Pic}(\Rs)\to H^2(\Rs;\Z)
\end{equation}
given by the first Chern class is an isomorphism \cite[Proposition~7.1.10]{Nemethi:NSS},
that is, all analytic line bundles are characterized topologically by the first Chern class.

\begin{remark}\label{rmk:link.qhs}
Let $L$ be the link of $(X,x)$ and recall from Subsection~\ref{ssec:qss} that $L$ is a spherical manifold. From Subsection~\ref{ssec:s3m.ccsn.xi} $L$ is a rational homology sphere, 
this also follows from the fact that $(X,x)$ is a rational singularity, see Laufer's Rationality Criterion \cite[Theorem~7.1.2-(iv)]{Nemethi:NSS} and \cite[3.4.2]{Nemethi:NSS}.
\end{remark}

\subsection{Flat and reflexive modules}

Let $(X,x)$ be the germ of a two-dimensional quotient singularity and $X$ some representative of such a germ.
An  $\Ss{X}$-module $M$ is \textit{indecomposable} if it is not the direct sum of two non trivial modules.
Let $\Homs_{\Ss{X}}(\cdot,\cdot)$ be the sheaf theoretical  Hom functor.
The \textit{dual} of an $\Ss{X}$-module $M$ is denoted by $M^{\smvee}:=\Homs_{\Ss{X}}(M,\Ss{X})$. An $\Ss{X}$-module $M$ is called \textit{reflexive} if the natural homomorphism from $M$ to $M^{\smvee \smvee}$ is an isomorphism (see \cite[Definition~5.1.12]{Ishii:Sing}).

Let $M$ be a coherent $\Ss{X}$-module and $\mathrm{End}_{\C}(M)$ the $\C$-vector space $\C$-linear maps.
Following~\cite{Gustavsen-Ile:RMNSSRLFG} a \textit{connection} on $M$ is an $\Ss{X}$-linear map
\begin{equation*}
	\nabla \colon \mathrm{Der}_{\C}(\Ss{X})\to \mathrm{End}_{\C}(M),
\end{equation*}
which for all $f\in \Ss{X}$, $m\in M$ and $D\in \mathrm{Der}_{\C}(\Ss{X})$ satisfies the \emph{Leibniz rule}
\begin{equation*}
	\nabla(D)(fm)=D(f)m+f \nabla(D)(m).
\end{equation*}
A connection $\nabla$ is called \emph{integrable or flat} if it is a $\C$-Lie algebra homomorphism.
An coherent $\Ss{X}$-module with a flat connection is called a \textit{flat module}.

Consider the following categories:
\begin{align*}
	\mathrm{Ref}_{X} &:= \text{the category of reflexive $\Ss{X}$-modules},\\
	\mathrm{Ref}^{\nabla}_{X} &:= \text{\parbox[t]{.85\textwidth}{the category of pairs $(M,\nabla)$ where $M$ is a reflexive $\Ss{X}$-module and $\nabla$ is an integrable connection,}}\\
	\mathrm{Rep}_{\pi_1(L)} &:= \text{the category of complex finite dimensional representations of $\pi_1(L)$}.
\end{align*}

By~\cite{Gustavsen-Ile:RMNSSRLFG} there is an equivalence of categories 
\begin{equation}\label{eq:Ref=Rep}
	\mathrm{Ref}^{\nabla}_{X} \cong \mathrm{Rep}_{\pi_1(L)}.
\end{equation}
Since $X$ is a quotient singularity, by Esnault~\cite{Esnault:RMQSS} the {\it forgetful functor} 
\begin{equation}\label{eq:RefN=Ref}
	\mathrm{Ref}^{\nabla}_{X} \to \mathrm{Ref}_{X},
\end{equation}
is an equivalence of categories. 

Let $\mathrm{Rep}^1_{\pi_1(L)}$ be the subcategory of $\mathrm{Rep}_{\pi_1(L)}$ of complex one dimensional representations of $\pi_1(L)$.
Let $\rho\in \mathrm{Obj}(\mathrm{Rep}^1_{\pi_1(L)})$, so it is a homomorphism $\rho \colon \pi_1(L) \to \mathrm{G}L(1,\C)=\C^*$ where $\C^*$ is the multiplicative subgroup of $\C$.
Since $\C^*$ is abelian,
\begin{equation}
	\label{eq:property.abelian}
	\Hom(\pi_1(L),\C^*) = \Hom(\pi_1(L)_{\mathrm{ab}},\C^*) = \Hom(H_1(L,\Z),\C^*).
\end{equation}
By the Universal Coefficient Theorem we get
\begin{equation}
	\label{eq:universal.coef.th}
	\Hom(H_1(L,\Z),\C^*) \cong H^1(L,\C^*).
\end{equation}
Hence, by~\eqref{eq:property.abelian} and~\eqref{eq:universal.coef.th}
$\rho \in H^1(L,\C^*)$. Replacing $\C^*$ by $\C/\Z$,  the exact sequence
\begin{equation*}
	0 \to \Z \to \C \to \C/\Z \cong \C^* \to 0,
\end{equation*}
induces  the following cohomological long exact sequence
\begin{equation}\label{eq:long.sequence.coh.Sigma}
	\dots \to H^1(L, \C) \to H^1(L, \C/\Z) \stackrel{\Phi}{\longrightarrow}	H^2(L, \Z) \to H^2(L, \C) \to \dots
\end{equation}
By Remark~\ref{rmk:link.qhs}, the link $L$ is a rational homology sphere and by~\eqref{eq:long.sequence.coh.Sigma}, the homomorphism $\Phi$ is an isomorphism
\begin{equation}\label{eq:Psi=c1}
\Phi=c_1\colon\mathrm{Rep}^1_{\pi_1(L)}\cong\Hom(\pi_1(L),\C^*)\cong H^1(L, \C/\Z)\to H^2(L, \Z)
\end{equation}
which corresponds to take the first Chern class $c_1(\rho)$ of a one-dimensional representation $\rho$ of $\pi_1(L)$.

\subsection{Full sheaves} Let $(X,x)$ be the germ of a two-dimensional quotient singularity and $ \pi \colon \Rs \to X $ be the minimal resolution. 
In \cite{Esnault:RMQSS} Esnault gave the definition of full sheaves.
\begin{definition}\label{def:full.sheaf}
An $\Ss{\Rs}$-module $\mathcal{M}$ is called \textit{full} if there is a reflexive $\Ss{X}$-module $M$ such that $\mathcal{M} \cong \left(\pi^* M\right)/\mathrm{tor}$.
We call $\mathcal{M}$ the \textit{full sheaf associated to $M$}. We denote by $\mathrm{Sh}_{\tilde{X}}^{\mathrm{Full}}$ the set of full sheaves.
A full sheaf is \textit{indecomposable} if it is not the direct sum of two non trivial full sheaves.
\end{definition}

\begin{remark}
By \cite[Lemma and definition~2.2]{Esnault:RMQSS} a full sheaf $\mathcal{M}$ is a locally free sheaf, that is, $\Rs$ can be covered by open sets $U$ for which
$\mathcal{M}|_U$ is isomorphic to the direc sum of $r$ copies of $\Ss{\Rs}|_U$ for some $r\in\mathbb{N}$ called the \textit{rank} of $\mathcal{M}$ denoted by $\rank\mathcal{M}$.
Taking the associated geometric vector bundle of $\mathcal{M}$ \cite[Exercise~II.5.18]{Hartshorne:AlgGeom} we can compute its first Chen class $c_1(\mathcal{M})\in H^{2}(\tilde{X};\Z)$.
Alternatively, if $\mathcal{M}$ has rank $r$, taking $r$ generic sections, one can define the first Chern class as the Poincaré dual of the homology class
defined by the discriminant of the $r$ sections \cite[p.~66]{Esnault:RMQSS}.
\end{remark}

\begin{remark}\label{rm:full.mcm.1to1}
Let $M$ be a reflexive $\Ss{X}$-module. Denote by $\mathcal{M} = \left(\pi^* M\right)/\mathrm{tor}$ the associated full sheaf. 
Then, by \cite[Lemma and definition~2.2]{Esnault:RMQSS} we have that  $\pi_* \mathcal{M}=M$. 
Thus, there is a one-to-one correspondence 
\begin{equation}\label{eq:full.reflex}
\mathrm{Sh}_{\tilde{X}}^{\mathrm{Full}}\longleftrightarrow  \mathrm{Obj}(\mathrm{Ref}_{X})
\end{equation}
between full sheaves and reflexive $\Ss{X}$-modules. 
\end{remark}

Using the following particular case of \cite[Theorem~6.19]{arciniegaetal:CCSC}, we can compute the invariant $\tilde{\xi}_\rho(D)$ on the minimal resolution.
\begin{theorem}\label{th:C2Res}
Let $(X,x)$ be the germ of a quotient surface singularity. Let $\pi \colon \Rs \to X$ be the minimal resolution. Let $(M,\nabla)$ be a flat reflexive $\Ss{X}$-module. Denote by $\rho \colon \pi_1(\Sigma) \to GL(n,\C)$ the representation associated to $(M,\nabla)$  by equivalence \eqref{eq:Ref=Rep}. Denote by $\mathcal{M}$ the corresponding full sheaf on $\Rs$. Then,
\begin{equation}\label{eq:xi.integral}
\tilde{\xi}_{\rho}(D)= \int_{\Rs} \left( \mathrm{ch} \, \mathcal{M} -n \right) \mathcal{T}(\Rs),
\end{equation}
where $\mathrm{ch}$ is the Chern character and $\mathcal{T}$ the Todd class. Moreover, if $\rho$ is topologically trivial, then by \eqref{eq:CCS2=xi} the second CCS-number $\widehat{c}_{\rho,2}([L])$ is given by \eqref{eq:xi.integral}.
\end{theorem}

Let $(X,x)$ be the germ of a quotient surface singularity and $\pi \colon \Rs \to X$ the minimal resolution. Let $M$ be a reflexive $\Ss{X}$-module. 
Denote by $\mathcal{M}$ the corresponding full sheaf on $\Rs$.
By the correspondence between full sheaves and reflexive $\Ss{X}$-modules and since $M$ admits a unique flat connection $\nabla$  by the 
equivalence of categories \eqref{eq:RefN=Ref}, we can see the invariant $\tilde{\xi}_\rho(D)$ as an invariant of $\mathcal{M}$ which we denote by
\begin{equation*}
	\tilde{\xi}(\mathcal{M}):=\tilde{\xi}_{\rho}(D),
\end{equation*}
where $\rho \colon \pi_1(\Sigma) \to GL(n,\C)$ is the representation associated to $(M,\nabla)$ by equivalence \eqref{eq:Ref=Rep}.

\subsection{Classification of reflexive modules on quotient surface singulairites} \label{sec.Classification}
In this section, we classify (almost) all the reflexive modules over quotient surface singularities. By Remark~\ref{rm:full.mcm.1to1}, it is enough to classify all the full sheaves. Moreover, we can restrict to the case of indecomposable full sheaves.

\begin{remark}\label{rmk:wunram.dificil}
If the singularity is a rational double point, then any indecomposable full sheaf is determined by its rank and its first Chern class. This was done by Artin and Verdier in \cite{Artin-Verdier:RMORDP}.
\end{remark}
	
Later, Esnault~\cite{Esnault:RMQSS} asked if the rank and the first Chern class are enough to classify the indecomposable full shaves for general quotient surface singularities. 
In \cite[pp.~597]{Wunram:RMOQSS} Wunram answered this question negatively. Over the quotient singularity $(X,x)=(\C^2/\mathbb{I}_7,0)$, where $\mathbb{I}_7\cong\mathsf{BI}\times\mathsf{C}_7$ (see Theorem~\ref{thm:fssu2}), Wunram constructed two different full shaves $\mathcal{M}_1$ and  $\mathcal{M}_2$ with the same rank and the same first Chern class:
\begin{equation}\label{eq:wunram.example}
\begin{split}
r=\rank \mathcal{M}_1 &= \rank \mathcal{M}_2,\\
c_1(\mathcal{M}_1) &= c_1(\mathcal{M}_2).
\end{split}
\end{equation}
Denote by $\mathcal{L}_1=\det \mathcal{M}_1$ and $\mathcal{L}_2=\det \mathcal{M}_2$ the determinant bundles of $\mathcal{M}_1$ and $\mathcal{M}_1$, respectively. Since, $c_1(\mathcal{M}_1)=c_1(\mathcal{L}_1)$ and $c_1(\mathcal{M}_2)=c_1(\mathcal{L}_2)$. Then, by~\eqref{eq:wunram.example} we get $c_1(\mathcal{L}_1)=c_1(\mathcal{L}_2)$. By Remark~\ref{rmk:quotient.isrational}, $(X,x)$ is a rational singularity. Since $\mathcal{L}_1$ and $\mathcal{L}_2$ are line bundles with the same first Chern class over a rational singularity, by the isomorphism \eqref{eq:Pic.Chern} they have to be the same, so $\mathcal{L}_1 \cong \mathcal{L}_2$. 
By the proof of~\cite[Proposition~2.8]{Esnault:RMQSS}, the full sheaves $\mathcal{M}_1$  and $\mathcal{M}_2$ appear as an extension of its determinant bundle and a trivial bundle, i.e.,
\begin{equation}\label{eq:wunram.example.2}
0 \to \Ss{\Rs}^{r-1} \to \mathcal{M}_i \to \mathcal{L}_i \to 0,\quad i=1,2.
\end{equation}
Note that~\eqref{eq:wunram.example.2} it is an exact sequence of vector bundles. In the differential category, every short exact sequence of differential vector bundles split. Therefore, as $C^{\infty}$-vector bundles, we get
\begin{equation*}
\mathcal{M}_1 \cong \Ss{\Rs}^{r-1} \oplus \mathcal{L}_1 \quad \text{and} \quad \mathcal{M}_2 \cong \Ss{\Rs}^{r-1} \oplus \mathcal{L}_2.
\end{equation*}
Since $\mathcal{L}_1 \cong \mathcal{L}_2$, we get that as $C^{\infty}$-vector bundles $\mathcal{M}_1$ is diffeomorphic to $\mathcal{M}_2$. Neverthelss, by Wunram~\cite[pp.~597]{Wunram:RMOQSS} we know that as holomorphic vector bundles $\mathcal{M}_1 \not \cong \mathcal{M}_2$. Thus, the classification problem of full shaves requires an invariant that determines the holomorphic structure.

In Section~\ref{sec:Class.FVB}, we classify all the indecomposable flat vector bundles over 3-dimensional spherical manifolds, except for one case.
We will use this result to prove the main theorem of this section but we need to do some minors modifications. 
The main reason is that $(X,x)$ is a singular space, and even more, the regular part $U$ of $X$ is a complex 2-dimensional space. 
It is true that $U$ and the link $L$ are homotopically equivalent but we need to guarantee that we do not loose the classification passing from the complex structure to the topological structure and vice-versa. Our main theorem is the following.

\begin{theorem}\label{th:main.class2}
Let $(X,x)$ be a quotient surface singularity and $\pi \colon \Rs \to X$ the minimal resolution. Let $\mathcal{M}_1$ and $\mathcal{M}_2$ be two full sheaves. Then $\mathcal{M}_1 \cong \mathcal{M}_2$ if and only if $\mathrm{rank}(\mathcal{M}_1)=\mathrm{rank}(\mathcal{M}_2)$, $c_1(\mathcal{M}_1)\cong c_1(\mathcal{M}_2)$ and $\tilde{\xi}(\mathcal{M}_1)=\tilde{\xi}(\mathcal{M}_2)$.
If $(X,x)=(\C^2/\mathbb{D}_{n,q},0)$ with $\gcd(m,2)=2$ we only consider $\mathcal{M}_1$ and $\mathcal{M}_2$ such that $\mathrm{rank}(\mathcal{M}_1)=\mathrm{rank}(\mathcal{M}_2)=1$.
\end{theorem}

\begin{proof}
First note that since the minimal resolution $\pi$ is a biholomorphism restricted to the complement of $E=\pi^{-1}(x)$, we can identify the link $L\subset X$ with its
inverse image $\pi^{-1}(L)$. Then, taking the inclusion $\iota\colon L\cong\pi^{-1}(L)\to \tilde{X}$ we obtain an induced homomorphism in cohomology
\begin{equation}\label{eq:iota}
\iota^*\colon H^2(\tilde{X}, \Z)\to H^2(L, \Z).
\end{equation}
If $\mathcal{M}_1 \cong \mathcal{M}_2$, the proof is trivial. Suppose that $\mathrm{rank}(\mathcal{M}_1)=\mathrm{rank}(\mathcal{M}_2)$, $c_1(\mathcal{M}_1)\cong c_1(\mathcal{M}_2)$ and $\tilde{\xi}(\mathcal{M}_1)=\tilde{\xi}(\mathcal{M}_2)$. Denote by $(M_{\rho_1},\nabla_{\rho_1})$ and $(M_{\rho_2},\nabla_{\rho_2})$ respectively, the associated flat reflexive $\Ss{X}$-modules associated to $\mathcal{M}_1$ and $\mathcal{M}_2$, and let $\rho_1 \colon \pi_1(L) \to \GLn{n_1}$ and  $\rho_2 \colon \pi_1(L) \to \GLn{n_2}$ be, respectively, the associated representations of the fundamental group of the link $L$, given by equivalence \eqref{eq:Ref=Rep}.
The geometric vector bundle associated to $\mathcal{M}_i$ has rank $\mathrm{rank}(\mathcal{M}_i)$ and its res\-triction to the link $L\cong\pi^{-1}(L)$ gives the vector bundle $V_{\rho_i}$ (see Subsection~\ref{ssec:CCScl}) corresponding to the representation $\rho_i$ by \eqref{eq:Flat.Rep}. Hence $\mathrm{rank}(\mathcal{M}_i)=n_i=\dim\rho_i$.
Since $\mathrm{rank}(\mathcal{M}_1)=\mathrm{rank}(\mathcal{M}_2)$ we have that $\dim(\rho_1)=\dim(\rho_2)$.

By Theorem~\ref{th:C2Res} the invariant $\tilde{\xi}(\mathcal{M}_i)$, $i=1,2$, can be computed on the link. Therefore, we just need to prove that
if  $c_1(\mathcal{M}_1)\cong c_1(\mathcal{M}_2)$, then $c_1(\rho_1)\cong c_1(\rho_2)$.
This follows by the commutative diagram

\begin{equation}
\xymatrix{
\mathrm{Sh}_{\tilde{X}}^{\mathrm{Full}}\ar@{<->}[r]^(0.42){\text{\eqref{eq:full.reflex}}}\ar[rd]_{c_1} &
\mathrm{Obj}(\mathrm{Ref}_{X})\ar@{<->}[r]^{\text{\eqref{eq:RefN=Ref}}} & \mathrm{Obj}(\mathrm{Ref}^{\nabla}_{X})\ar@{<->}[r]^(0.5){\text{\eqref{eq:Ref=Rep}}} &
\mathrm{Rep}_{\pi_1(L)}\ar[ld]^{c_1}\ar[dd]^{\det}\\
 & H^2(\tilde{X}, \Z)\ar[r]^{\iota^*}_{\text{\eqref{eq:iota}}}& H^2(L, \Z) & \\
 & & &\mathrm{Rep}^1_{\pi_1(L)}\cong H^1(L;\C/\Z)\ar[lu]^{\cong}_{\Phi=c_1\text{ \eqref{eq:Psi=c1}}}
}
\end{equation}
We have that $\mathcal{M}_1,\mathcal{M}_2\in \mathrm{Sh}_{\tilde{X}}^{\mathrm{Full}}$. 
By hypo\-the\-sis $c_1(\mathcal{M}_1)\cong c_1(\mathcal{M}_2)$ and by the commutativity of the upper part of the diagram we have that $c_1(\rho_1)=c_1(\rho_2)$.
Moreover, by the commutativity of the triangle on the right side of the diagram and by the isomorphism $\Phi$ given in \eqref{eq:Psi=c1} we have that $\det\rho_1=\det\rho_2$.
Furthermore, by \eqref{eq:ccs-c1} we have the equality of the first Cheeger-Chern-Simons classes $\widehat{c}_{\rho_1,1}=\widehat{c}_{\rho_2,1}$ and their
corresponding first CCS-numbers $\widehat{c}_{\rho_1,1}(\nu_i)=\widehat{c}_{\rho_2,1}(\nu_i)$ with $\nu_i$ generator(s) of $H_1(L;\Z)$.

On the other hand, also from $\det\rho_1=\det\rho_2$ we have that $\tilde{\xi}_{\det\rho_1}(D)=\tilde{\xi}_{\det\rho_2}(D)$,
hence by Theorem~\ref{thm:rhs}-\eqref{it:xi.non.triv} we have that 
\begin{equation*}
\widehat{c}_{\rho_1,2}([L])=\tilde\xi_{\rho_1}(D)-\tilde\xi_{\det(\rho_1)}(D)=	\tilde\xi_{\rho_2}(D)-\tilde\xi_{\det(\rho_2)}(D)=\widehat{c}_{\rho_1,2}([L]).
\end{equation*}

Therefore, we have proved that
\begin{equation*}
\dim(\rho_1)=\dim(\rho_2),\quad \widehat{c}_{\rho_1,1}(\nu_i)=\widehat{c}_{\rho_2,1}(\nu_i),\quad \text{and}\quad \widehat{c}_{\rho_1,2}([L])=\widehat{c}_{\rho_1,2}([L]).
\end{equation*}
Then, by Theorem~\ref{th:main.class} we have $\rho_1=\rho_2$, which implies that $\mathcal{M}_1\cong\mathcal{M}_2 $.

The case $\mathbb{D}_{n,q}$ with $\gcd(m,2)=2$, by Theorem~\ref{thm:D.T}-\eqref{it:T.A} corresponds to the case $\Gamma=D_{2^k(2r+1)}\times\mathsf{C}_l$ with $\gcd(2^k(2r+1),l)=1$
in Theorem~\ref{th:main.class}.
\end{proof} 

Theorem~\ref{th:main.class2} is equivalent to Theorem~\ref{Thm3} in the Introduction.

\appendix

\section{Chern classes of the tensor product of representations}\label{app:tensor}

In the proof of Proposition~\ref{prop:G.product.class} we need formulas for the first and second Chern classes of the tensor product of two representations. In this appendix
we show how to get these formulae.

By \cite[Appendix~(8)]{Atiyah-PMIHES} if $x_1,\dots,x_n,y_1,\dots,y_n$ are two sets of indeterminates  with elementary symmetric functions $a_i,b_i$ respectively, we can define polynomials $Q_k$ by the formula
\begin{equation}\label{eq:Qk}
\prod_{\substack{1\leq i\leq n\\1\leq j\leq m}}\bigl(1+t(x_i+y_j)\bigr)=\sum_k Q_k(a_1,\dots,a_n,b_1,\dots,b_m)t^k
\end{equation}
where $t$ is an indeterminate.

Let $\rho$ and $\sigma$ be two representations of ranks $n$ and $m$ respectively.
The Chern classes of the tensor product $\rho\otimes\sigma$ are then given by
\begin{equation}\label{eq:cc.tenprod.formula}
c_k(\rho\otimes\sigma)=Q_k(c_1(\rho),\dots,c_n(\rho),c_1(\sigma),\dots,c_m(\sigma.))
\end{equation}
Note that if $\dim\rho=\dim\sigma=1$, formula \eqref{eq:cc.tenprod.formula} gives $ c_1(\rho\otimes\sigma)=c_1(\rho)+c_1(\sigma)$.

\subsection{Our case}

We are interested in the case when $n=1,2,3,4,5,6$, $m=1$ and $k=1,2$. But we do it for any $n$.

We need to compute the product \eqref{eq:Qk} with $m=1$
\begin{equation}
(1+tx_1+ty_1)(1+tx_2+ty_1)\cdots(1+tx_n+ty_1).
\end{equation}
Let us analyze how are the terms of the different degrees of $t$.
The terms of the product are obtained choosing one of the three terms in each factor and multipliying them.
\begin{description}
\item[Degree $0$] The only term of degree zero is $1$ and it is obtained taking \textbf{all} the $1$'s in each factor.
\item[Degree $1$] The terms of degree one are obtained choosing $(n-1)$ $1$'s in $(n-1)$ factors and one term different from $1$ in the remaining factor.
If we choose the term different from $1$ a term with an $x_i$ we obtain
\begin{equation}\label{eq:deg1.1}
(x_1+\dots+x_n)t.
\end{equation}
If we choose the term different from $1$ the term $ty_1$ we get
\begin{equation}\label{eq:deg1.2}
ny_1t.
\end{equation}
From \eqref{eq:deg1.1} and \eqref{eq:deg1.2} the sum of all the term of degree one is
\begin{equation}\label{eq:terms1}
\bigl((x_1+\dots+x_n)+ny_1\bigr)t.
\end{equation}
\item[Degree $2$] The terms of degree two are obtained choosing $(n-2)$ $1$'s in $(n-2)$ factors and one term different from $1$ in the two remaining factors.
If we choose the terms different from $1$ terms $x_i$ and $x_j$, with $1\leq i,j\leq n$ and $i\neq j$, we obtain
\begin{equation}\label{eq:deg2.1}
(x_1x_2+\dots+x_1x_n+x_2x_3+\dots+x_{n-1}x_n)t^2.
\end{equation}
If we choose the terms different from $1$ the term $x_i$ and $y_1$, we have that $1\leq i\leq n$ and $y_1$ can be in any of the other $(n-1)$ factors, we obtain
\begin{equation}\label{eq:deg2.2}
(n-1)(x_1y_1+\dots+x_ny_1)t^2.
\end{equation}
If we choose the terms different from $1$ two terms $y_1$, we can choose them of $\binom{n}{2}$ ways, so we get
\begin{equation}\label{eq:deg2.3}
\binom{n}{2}y_1^2t^2.
\end{equation}
From \eqref{eq:deg2.1}, \eqref{eq:deg2.2} and \eqref{eq:deg2.3} the sum of all the term of degree two is
\begin{equation}\label{eq:terms2}
\left((x_1x_2+\dots+x_1x_n+x_2x_3+\dots+x_{n-1}x_n)+(n-1)(x_1+\dots+x_n)y_1+\binom{n}{2}y_1^2\right)t^2.
\end{equation}
The elementary symmetric functions on the $x_i$ are
\begin{align}
a_1&=x_1+\dots+x_n\label{eq:a1}\\
a_2&=x_1x_2+\dots+x_1x_n+x_2x_3+\dots+x_{n-1}x_n\label{eq:a2}
\end{align}
and $b_1=y_1$.

We have that $c_1(\rho)=a_1$, $c_2(\rho)=a_2$ and $c_1(\sigma)=b_1$.
Hence from the terms of degree one in \eqref{eq:terms1}, \eqref{eq:a1} and \eqref{eq:cc.tenprod.formula} we get
\begin{equation}
c_1(\rho\otimes\sigma)=c_1(\rho)+nc_1(\sigma).
\end{equation}
From the terms of degree two in \eqref{eq:terms2}, \eqref{eq:a1}, \eqref{eq:a2} and \eqref{eq:cc.tenprod.formula} we get
\begin{equation}
c_2(\rho\otimes\sigma)=c_2(\rho)+(n-1)c_1(\rho)\cdot c_1(\sigma)+\binom{n}{2}c_1(\sigma)^2.
\end{equation}
\end{description}

\bibliography{CSS-references}
\bibliographystyle{plain}
\end{document}